\documentclass[a4paper,10pt]{article}

\usepackage{geometry}
\usepackage{epsfig}
\usepackage{amsmath}
\usepackage{amssymb}
\usepackage{amsthm}
\usepackage{mathrsfs}
\usepackage{color}
\usepackage{amscd}
\usepackage{hyperref}

\usepackage{amsthm}
\newtheorem{theorem}{Theorem}[section]
\newtheorem{definition}[theorem]{Definition}
\newtheorem{lemma}[theorem]{Lemma}
\newtheorem{proposition}[theorem]{Proposition}
\newtheorem{corollary}[theorem]{Corollary}
\newtheorem{remark}[theorem]{Remark}
\newtheorem{example}[theorem]{Example}

\usepackage{amsfonts}

\newcommand{\hh}{{\mathbb{H}}}

\newcommand{\cc}{{\mathbb{C}}}
\newcommand{\rr}{{\mathbb{R}}}

\newcommand{\nn}{{\mathbb{N}}}

\newcommand{\J}{\mathcal{J}}

\newcommand{\s}{{\mathbb{S}}}

\newcommand{\A}{\mathcal{A}}
\newcommand{\I}{\mathcal{I}}
\renewcommand{\L}{\mathcal{L}}

\newcommand{\B}{\mathcal{B}}

\newcommand{\M}{\mathcal{M}}
\newcommand{\Q}{\mathcal{Q}}

\newcommand\re{\operatorname{Re}}
\newcommand\im{\operatorname{Im}}
\newcommand\tr{\operatorname{tr}}
\newcommand\rank{\operatorname{rank}}
\newcommand\aut{\operatorname{Aut}}
\newcommand\antiaut{\operatorname{Anti}}
\newcommand\lin{\operatorname{Lin}}
\newcommand\jac{\operatorname{Jac}}

\title{\bf Which Fueter-regular functions are holomorphic?}

\author{Alessandro Perotti\\
\small Dipartimento di Matematica, Universit\`a di Trento\\ 
\small Via Sommarive 14, I-38123 Povo Trento, Italy\\
\small alessandro.perotti@unitn.it
\thanks{Partly supported by: GNSAGA INdAM; Progetto ``Teoria delle funzioni ipercomplesse e applicazioni'' Universit\`a di Firenze.}\\
\and
Caterina Stoppato
\\ 
\small Dipartimento di Matematica e Informatica ``U. Dini'', Universit\`a di Firenze \\
\small Viale Morgagni 67/A, I-50134 Firenze, Italy\\
\small caterina.stoppato@unifi.it
\thanks{Partly supported by: GNSAGA INdAM; Progetto ``Teoria delle funzioni ipercomplesse e applicazioni'' Universit\`a di Firenze; PRIN 2022 ``Real and complex manifolds: geometry and holomorphic dynamics'' MIUR; Finanziamento Premiale ``Splines for accUrate NumeRics: adaptIve models for Simulation Environments'' INdAM.}
}
\date{  }

\begin{document}

\maketitle
\thanks{\footnotesize \it This work stems from a question posed by Massimo Tarallo (Universit\`a di Milano), who tragically passed away in 2019. The authors gratefully acknowledge his contribution to this work. Massimo is very much missed by those who knew him.}


\begin{abstract}
We provide a classification of Fueter-regular quaternionic functions $f$ in terms of the degree of complex linearity of their real differentials $df$. Quaternionic imaginary units define orthogonal almost-complex structures on the tangent bundle of the quaternionic space by left or right multiplication. Every map of two complex variables that is holomorphic with respect to one of these structures defines a Fueter-regular function. We classify the differential $df$ of a Fueter-regular function $f$, roughly speaking, in terms of how many choices of complex structures make $df$ complex linear. It turns out that, generically, $f$ is not holomorphic with respect to any choice of almost-complex structures. In the special case when it is indeed holomorphic, generically there is a unique choice of almost-complex structures making it holomorphic. The case of holomorphy with respect to several choices of almost-complex structures is limited to conformal real affine transformations or constants.
\end{abstract}


\section{Introduction}\label{sec:introduction}

Let $\hh=\rr+i\rr+j\rr+k\rr$ denote the real $*$-algebra of quaternions. Finding an analog, in the quaternionic setting, of the class of holomorphic complex-valued functions of one complex variable is a very natural question, which received a first positive answer in Rudolf Fueter's works~\cite{fueter1,fueter2,fueter4,fueter3} (see~\cite{sudbery} for a modern exposition). Let $U$ be a connected open subset of $\hh$ and let $f:U\to\hh,\ x=x_0+ix_1+jx_2+kx_3\mapsto f(x)$ be a $C^1$ function: then $f$ is called \emph{right Fueter-regular} if it belongs to the kernel of the \emph{right Cauchy-Riemann-Fueter operator} 
\[\bar\partial_r:=\frac12\left(\frac{\partial}{\partial x_0}+\frac{\partial}{\partial x_1}i+\frac{\partial}{\partial x_2}j+\frac{\partial}{\partial x_3}k\right)\,.\]
Right Fueter-regularity is an interesting refinement of harmonicity in the four real variables $x_0,x_1,x_2,x_3$: indeed, the real Laplacian can be factorized as
\[\Delta=4\bar\partial_r\partial_r=4\partial_r\bar\partial_r\,,\quad\partial_r:=\frac12\left(\frac{\partial}{\partial x_0}-\frac{\partial}{\partial x_1}i-\frac{\partial}{\partial x_2}j-\frac{\partial}{\partial x_3}k\right)\,.\]
In particular, right Fueter-regular functions are real analytic. Similar considerations apply to the class of \emph{left Fueter-regular} functions, defined as the kernel of \emph{left Cauchy-Riemann-Fueter operator} 
\[\bar\partial_l:=\frac12\left(\frac{\partial}{\partial x_0}+i\frac{\partial}{\partial x_1}+j\frac{\partial}{\partial x_2}+k\frac{\partial}{\partial x_3}\right)\,.\]
For more details on Cauchy-Riemann-Fueter operators and Fueter-regular functions, we refer the reader to~\cite{soucek,sudbery}. A bijective right Fueter-regular function $f:U\to f(U)$ is called \emph{RL-biregular} if $f^{-1}$ is left Fueter-regular and \emph{RR-biregular} if $f^{-1}$ is right Fueter-regular. Similarly, a left Fueter-regular function is called \emph{LR-biregular} if it admits a right Fueter-regular inverse, or \emph{LL-biregular} if it admits a left Fueter-regular inverse.
About biregularity, we refer the reader to~\cite{krolikowskiporter}. A rich quaternionic function theory has thus been developed, including integral and series representations, not only in one variable but also in several quaternionic variables (see, e.g.,~\cite{pertici}). Fueter-regularity has been generalized to the notion of \emph{q-holomorphic} map between hypercomplex manifolds, see~\cite{joyce}. In particular, the class of q-holomorphic self-maps of $\hh$ coincides with the class of right Fueter-regular functions. Fueter-regularity has also been generalized to the notion of \emph{quaternionic map} between hyperk\"ahler manifolds, see~\cite{chenli}. Analogs of Fueter's theory over other alternative algebras, such as the the theory of \emph{monogenic} functions with values in a real Clifford algebra $C\ell(0,n)$ with variable in the paravector subspace $\rr^{n+1}$, have also been extremely successful. We refer the interested reader to the monographs~\cite{librosommen,librocnops,librogurlebeck2}.

Within this broad panorama, it is a matter of some interest to relate Fueter-regularity in one quaternionic variable and holomorphy in two complex variables. While it is easy to find a copy of the class of holomorphic functions $(f_1,f_2):\cc^2\to\cc^2$ among right Fueter-regular functions as $f(z_1+z_2j)=f_1(z_1,z_2)+jf_2(z_1,z_2)$, investigating holomorphy with respect to (possibly nonconstant) orthogonal almost-complex structures of general Fueter-regular functions requires quite some work. The article~\cite{perottibiregular} proved that all biregular functions are biholomorphic, using tools from~\cite{chenli,joyce}. A new, rather powerful approach was introduced in~\cite{tarallo} and successfully applied to regular real linear maps $\hh\to\hh$. The author thought of them as real differentials $df_p$ of Fueter-regular functions $f$ and proved partial results for these functions, too. The same work~\cite{tarallo} related regularity of real linear maps to a suitable notion of quaternionic linearity, introducing the concept of \emph{size} (recalled in the forthcoming Subsection~\ref{subsec:regularlinearmaps}). The aim of the present work is to achieve a complete classification of right Fueter-regular functions, after completely classifying regular real linear maps in terms of size.

The paper is structured as follows. Section~\ref{sec:preliminaries} contains preliminary material about real quaternions, quaternionic spaces, linear maps between quaternionic spaces, regularity and conformality in one quaternionic variable. Section~\ref{sec:regularlinearmaps} presents some new results about regularity, quaternionic linearity and complex linearity of real linear maps. After some constructions and results valid for general real linear maps, Theorem~\ref{thm:lineargg} is a first classification of regular real linear maps in terms of size, covering only a specific type of complex linearity. Then a complete classification of regular real linear maps is provided in Theorem~\ref{thm:classificationlinear}. Section~\ref{sec:regularfunctions} focuses mostly on the classification of right Fueter-regular functions, achieved in Theorem~\ref{thm:classificationfueter}, but also provides a simple differential criterion to establish holomorphy, Theorem~\ref{thm:differentialcriterionholomorphy}. The final Appendix is devoted to proving two theorems, whose proofs are too technical to be interspersed within Sections~\ref{sec:regularlinearmaps} and~\ref{sec:regularfunctions}.


\section{Preliminaries}\label{sec:preliminaries}
\subsection{Real quaternions and quaternionic spaces}\label{subsec:quaternions}

The symbols $\cc,\hh$ will denote the real $^*$-algebras of complex numbers and quaternions, respectively. Starting out from the real field $\rr$, they arise from so-called Cayley-Dickson construction (see~\cite{ebbinghaus,libroward}):
\begin{itemize}
\item $\cc=\rr+i\rr$, $(\alpha+i\beta)(\gamma+i\delta)=\alpha\gamma-\beta\delta+i(\alpha\delta+\beta\gamma)$, $\overline{\alpha+i\beta}=\alpha-i\beta\ \forall\,\alpha,\beta,\gamma,\delta\in\rr$,
\item $\hh=\cc+j\cc$, $(\alpha+j\beta)(\gamma+j\delta)=\alpha\gamma-\bar\beta\delta+j(\bar\alpha\delta+\beta\gamma)$, $\overline{\alpha+j\beta}=\bar\alpha-j\beta\ \forall\,\alpha,\beta,\gamma,\delta\in\cc$.
\end{itemize}
Each of them is a real \emph{algebra} (a real vector space with a bilinear multiplicative operation) and is \emph{unitary} (has a multiplicative neutral element $1$). In either algebra, the subalgebra generated by $1$ is identified with $\rr$. Famously, $\hh$ is associative but not commutative: its \emph{center} $\{r\in\hh\, |\, rx=xr\ \forall\,x\in \hh\}$ is $\rr$.
The Cayley-Dickson construction also endows each of $\cc,\hh$ with a \emph{$^*$-algebra} structure: $cj: x\mapsto\bar x$ is a \emph{$^*$-involution} (it is a real linear transformation and fulfills the properties $\overline{(\bar x)}=x,\overline{xy}=\bar y\bar x$ for every $x,y$ and $\bar x=x$ for every $x \in \rr$). For each $x\in\hh$, the \emph{real part} $\re(x)$ and the \emph{imaginary part} $\im(x)$ are defined as $\re(x):=(x+\bar x)/2$ and $\im(x):=x-\re(x)=(x-\bar x)/2$. In $\hh\simeq\rr^4$,  the expression $\re(x\bar y)$ coincides with the standard scalar product $\langle x,y\rangle$. The Euclidean norm $\Vert x\Vert=\sqrt{x\bar x}$ is also called the \emph{modulus} of $x$ and denoted by $|x|$. The equality $|xy| = |x|\,|y|$ holds.

It is well known that every element $x$ of $\cc^*:=\cc\setminus\{0\}$ or $\hh^*:=\hh\setminus\{0\}$ has a multiplicative inverse, namely $x^{-1} = |x|^{-1} \bar x = \bar x\, |x|^{-1}$. For all $x,y\neq 0$, the equality $(xy)^{-1} = y^{-1}x^{-1}$ holds. Thus, each of the algebras $\cc$ and $\hh$ is a division algebra and has no proper zero divisors. 
A celebrated theorem by Frobenius states that $\rr,\cc,\hh$ are the only (finite-dimensional) associative division algebras. The Skolem-Noether theorem establishes that all (real algebra) automorphisms of $\hh$ are inner: $\aut(\hh)=\{\vartheta_p : p\in\hh^*\}$, where
\[\vartheta_p:\hh\to\hh\quad x\mapsto pxp^{-1}\,.\]
The automorphism $\vartheta_p$ is the identity map $id:\hh\to\hh$ if, and only if, $p\in\rr$. When restricted to $\im(\hh)\simeq\rr^3$, $\vartheta_p$ is a rotation about $\im(p)$. The rotation angle equals $\pi$ exactly when $p\in\im(\hh)$, a case when $\vartheta_p$ is called an \emph{imaginary automorphism}. The antiautomorphisms of $\hh$ are obtained by composing the automorphisms with conjugation $cj:\hh\to\hh, x\mapsto\bar x$: indeed, $\antiaut(\hh)=\{\bar\vartheta_p : p\in\hh^*\}$, where $\bar\vartheta_p(x)=\overline{pxp^{-1}}=p\bar xp^{-1}=\vartheta_p(\bar x)$ for all $p\in\hh^*$ and all $x\in\hh$. We point out that $\vartheta_p\circ\vartheta_q=\vartheta_{pq}=\bar\vartheta_p\circ\bar\vartheta_q$ and $\vartheta_p\circ\bar\vartheta_q=\bar\vartheta_{pq}=\bar\vartheta_p\circ\vartheta_q$ for all $p,q\in\hh^*$.

We adopt the standard notation for the $2$-sphere of quaternionic \emph{imaginary units}
\begin{equation} \label{eq:s_A}
\s=\s^2:=\{x \in \hh\, | \, \re(x)=0, |x|=1\} = \{x \in \hh \, |\, x^2=-1\}\,.
\end{equation}
The $^*$-subalgebra generated by any $g \in \s$, i.e., $\cc_g=\mathrm{span}(1,g)$, is $^*$-isomorphic to the complex field $\cc$. The equality
\begin{equation} \label{eq:slice}
\bigcup_{g \in \s}\cc_g = \hh
\end{equation}
holds true, while $\cc_g \cap \cc_h=\rr$ for every $g,h \in \s$ with $g \neq \pm h$. As a consequence, every element $x$ of $\hh \setminus \rr$ can be expressed as follows: $x=\alpha+\beta g$, where $\alpha \in \rr$ is uniquely determined by $x$, while $\beta \in \rr$ and $g \in \s$ are uniquely determined by $x$, but only up to sign. If $x \in \rr$, then $\alpha=x$, $\beta=0$ and $g$ can be chosen arbitrarily in $\s$. The $2$-sphere (or singleton) $\alpha+\beta\s$ is the orbit of $x$ both under the action of $\aut(\hh)$ and under the action of $\antiaut(\hh)$.

In this paper, we will often work with ordered bases $\B=(e_0,e_1,e_2,e_3)$ of $\hh$. When $\B$ is orthonormal and $e_0$ is the quaternion $1$, then: either $e_3=e_1e_2$ and $\B$ is \emph{positively oriented}; or $e_3=e_2e_1$ and $\B$ is \emph{negatively oriented}. Choosing a positively oriented orthonormal $\B=(e_0,e_1,e_2,e_3)$ with $e_0=1$ is the same as choosing an identification between $\cc^2$ and $\cc_{e_1}+\cc_{e_1}e_2=\hh$.

We consider on $\cc$ and $\hh$ the natural Euclidean topology and differential structure as a finite-dimensional real vector space. Similarly, we will adopt on $\cc$ and $\hh$ the natural structure of real analytic manifold. The relative topology on each $\cc_g$ with $g \in \s$ clearly agrees with the topology determined by the natural identification between $\cc$ and $\cc_g$.

A \emph{left quaternionic space} is a left module $X$ over the unital ring $\hh$: namely, a set $X$ with an addition and a left multiplication by quaternions such that, for all $x,x'\in X$ and all $p,q\in\hh$, the equalites $q(x+x')=qx+qx', (p+q)x=px+qx, p(qx)=(pq)x, 1x=x$ all hold true. Of course such an $X$ is also a real vector space. Similarly, a \emph{right quaternionic space} is a right module $X$ over the unital ring $\hh$. If $X$ is both a left and a right quaternionic space, it is called a \emph{bilateral quaternionic space}. For every $n\in\nn$ with $n\geq1$, $X=\hh^n$ is an example of bilateral quaternionic space. We will sometimes write $\hh^n_l$ for $\hh^n$ endowed with the natural left multiplication by quaternions and $\hh^n_r$ for $\hh^n$ endowed with the natural right multiplication by quaternions.

\subsection{Linear maps between quaternionic spaces}\label{subsec:linear}

Let $X,X'$ be given left quaternionic spaces and $Y,Y'$ be given right quaternionic spaces. We will survey all notions of linearity for maps $X\to X'$ or $X\to Y$ considered in~\cite[\S3]{tarallo} and provide a few glimpses of maps $Y\to X$ or $Y\to Y'$. These last two cases are covered in full detail in cited text~\cite[\S3]{tarallo}.

The left quaternionic space of real linear maps $X\to X'$ is denoted by $\lin^\rr(X,X')$, or by $\lin^\rr(X)$ when $X'=X$. Similarly: the right quaternionic space of real linear maps $X\to Y$ is denoted by $\lin^\rr(X,Y)$; the left quaternionic space of real linear maps $Y\to X$ is denoted by $\lin^\rr(Y,X)$; and the right quaternionic space of real linear maps $Y\to Y'$ is denoted by $\lin^\rr(Y,Y')$, or by $\lin^\rr(Y)$ when $Y'=Y$. All of these spaces are, of course, also real vector spaces. In the special case when $X=Y$ is a bilateral quaternionic space, then $\lin^\rr(X,Y)=\lin^\rr(Y,X)$ is a bilateral quaternionic space denoted by $\lin^\rr(X)$.

Fix $g,h\in\s$. An element $\Gamma\in\lin^\rr(X,X')$ is called \emph{complex linear of type $(g,h)$} if $\Gamma(gx)=h\Gamma(x)$ for all $x\in X$. The real vector space of complex linear maps of type $(g,h)$ from $X$ to $X'$ is denoted by $\lin^\cc_{gh}(X,X')$. Similarly, an element $\Lambda\in\lin^\rr(X,Y)$ is called \emph{complex linear of type $(g,h)$} if
\[\Lambda(gx)=\Lambda(x)h\]
for all $x\in X$. The real vector space of complex linear maps of type $(g,h)$ from $X$ to $Y$ is denoted by $\lin^\cc_{gh}(X,Y)$. The  following properties are well known for all $g,h\in\s$:
\begin{itemize}
\item $\lin^\cc_{gh}(X,X')=\lin^\cc_{(-g)(-h)}(X,X')$;
\item $\lin^\cc_{gh}(X,Y)=\lin^\cc_{(-g)(-h)}(X,Y)$;
\item $\lin^\cc_{gh}(X,X')\cap\lin^\cc_{g'h'}(X,X')\neq\{0\} \Rightarrow \langle g,g'\rangle=\langle h,h'\rangle$;
\item $\lin^\cc_{gh}(X,Y)\cap\lin^\cc_{g'h'}(X,Y)\neq\{0\} \Rightarrow \langle g,g'\rangle=\langle h,h'\rangle$;
\item $\Gamma\in\lin^\cc_{gg'}(X,X'),\Lambda\in\lin^\cc_{g'h}(X',Y) \Rightarrow \Lambda\circ\Gamma\in\lin^\cc_{gh}(X,Y)$;
\item $\lin^\rr(X,X')=\lin^\cc_{gh}(X,X')\oplus\lin^\cc_{g(-h)}(X,X')$;
\item $\lin^\rr(X,Y)=\lin^\cc_{gh}(X,Y)\oplus\lin^\cc_{g(-h)}(X,Y)$;
\item $\dim_\rr(X),\dim_\rr(X')$ finite $\Rightarrow \rank(\Gamma)\in2\nn$ for all $\Gamma\in\lin^\cc_{gh}(X,X')$;
\item $\dim_\rr(X),\dim_\rr(Y)$ finite $\Rightarrow \rank(\Lambda)\in2\nn$ for all $\Lambda\in\lin^\cc_{gh}(X,Y)$;
\item $\Gamma\in\lin^\cc_{gh}(X,X'),\dim_\rr(X)=\dim_\rr(X')=\rank(\Gamma)$ finite $\Rightarrow \Gamma$ preserves orientation;
\item $\Lambda\in\lin^\cc_{gh}(X,Y),\dim_\rr(X)=\dim_\rr(Y)=\rank(\Lambda)$ finite $\Rightarrow \Lambda$ reverses orientation.
\end{itemize}

The article~\cite{tarallo} studies to which extent complex linearity admits a quaternionic analog. For any given $\varphi:\hh\to\hh$, called a \emph{reference map}, the symbol
\[\lin^\hh_\varphi(X,X')\]
denotes the real vector space of \emph{$\varphi$-linear} maps $X\to X'$, i.e., maps $\Gamma:X\to X'$ such that
\begin{equation}\label{eq:LL-linear}
\Gamma(qx)=\varphi(q)\Gamma(x)
\end{equation}
for all $x\in X$ and all $q\in\hh$. Again, $\lin^\hh_\varphi(X)$ stands for $\lin^\hh_\varphi(X,X')$ when $X'=X$. This is a rather broad generalization of the usual notion of \emph{quaternionic linear} map, which is subsumed as the special case when the reference map $\varphi$ is the identity map $id:\hh\to\hh$. In other words, $\lin^\hh_{id}(X,X')$ is the real vector space of quaternionic linear maps $X\to X'$. Similarly, the symbol $\lin^\hh_\varphi(X,Y)$ denotes the real vector space of \emph{$\varphi$-linear} maps $X\to Y$, i.e., maps $\Lambda:X\to Y$ such that $\Lambda(qx)=\Lambda(x)\varphi(q)$ for all $x\in X$ and all $q\in\hh$. Again, choosing $\varphi$ to be conjugation $cj$ allows to recover the usual real vector space of \emph{quaternionic antilinear} maps $X\to Y$ as $\lin^\hh_{cj}(X,Y)$. Finally: the symbol $\lin^\hh_\varphi(Y,Y')$ denotes the real vector space of \emph{$\varphi$-linear} maps $Y\to Y'$, i.e., maps $\Gamma:Y\to Y'$ such that $\Gamma(yq)=\Gamma(y)\varphi(q)$ for all $y\in Y$ and all $q\in\hh$; the symbol $\lin^\hh_\varphi(Y)$ stands for $\lin^\hh_\varphi(Y,Y')$ when $Y'=Y$; the symbol $\lin^\hh_\varphi(Y,X)$ denotes the real vector space of \emph{$\varphi$-linear} maps $Y\to X$, i.e., maps $\Lambda:Y\to X$ such that $\Lambda(yq)=\varphi(q)\Lambda(y)$ for all $y\in Y$ and all $q\in\hh$.

The cited~\cite{tarallo} proves the following properties, valid for any choice of reference maps $\vartheta,\vartheta',\varphi,\varphi'$, of imaginary units $g,g',h,h'\in\s$ and of quaternions $p,q\in\hh^*$ that are linearly independent over $\rr$:
\begin{itemize}
\item $\lin^\hh_\vartheta(X,X')\neq\{0\}\Leftrightarrow\vartheta\in\aut(\hh)\Leftrightarrow\lin^\hh_\vartheta(Y,Y')\neq\{0\}$;
\item $\lin^\hh_\varphi(X,Y)\neq\{0\}\Leftrightarrow\varphi\in\antiaut(\hh)\Leftrightarrow\lin^\hh_\varphi(Y,X)\neq\{0\}$;
\item $\vartheta\neq\vartheta'\Rightarrow\lin^\hh_\vartheta(X,X')\cap\lin^\hh_{\vartheta'}(X,X')=\{0\}=\lin^\hh_\vartheta(Y,Y')\cap\lin^\hh_{\vartheta'}(Y,Y')$;
\item $\varphi\neq\varphi'\Rightarrow\lin^\hh_\varphi(X,Y)\cap\lin^\hh_{\varphi'}(X,Y)=\{0\}=\lin^\hh_\varphi(Y,X)\cap\lin^\hh_{\varphi'}(Y,X)$;
\item $\Gamma\in\lin^\hh_\vartheta(X,X'),\Lambda\in\lin^\hh_\varphi(X',Y)\Rightarrow \Lambda\circ\Gamma\in\lin^\hh_{\varphi\circ\vartheta}(X,Y)$;
\item $\Lambda\in\lin^\hh_\varphi(X,Y),\Gamma\in\lin^\hh_\vartheta(Y,Y')\Rightarrow \Gamma\circ\Lambda\in\lin^\hh_{\vartheta\circ\varphi}(X,Y')$;
\item $\Lambda\in\lin^\hh_\varphi(X,Y),\Lambda'\in\lin^\hh_{\varphi'}(Y,X')\Rightarrow \Lambda'\circ\Lambda\in\lin^\hh_{\varphi'\circ\varphi}(X,X')$;
\item $\lin^\hh_\vartheta(X,X')$ is a real vector subspace of $\lin^\rr(X,X')$; moreover, if $\vartheta\in\aut(\hh)$, then $\lin^\hh_\vartheta(X,X')$ is a real vector subspace of $\lin^\cc_{g\,\vartheta(g)}(X,X')$;
\item $\lin^\hh_\varphi(X,Y)$ is a real vector subspace of $\lin^\rr(X,Y)$; moreover, if $\varphi\in\antiaut(\hh)$, then $\lin^\hh_\varphi(X,Y)$ is a real vector subspace of $\lin^\cc_{g\,\varphi(g)}(X,Y)$;
\item $\Gamma\in\lin^\hh_{\vartheta_p}(X,X'),\Gamma'\in\lin^\hh_{\vartheta_q}(X,X') \Rightarrow \Gamma,\Gamma',\Gamma+\Gamma'\in\lin^\cc_{gh}(X,X')$ with $g=\im\left(\frac{q^{-1}p}{|q^{-1}p|}\right)$, $h=\vartheta_p(g)=\vartheta_q(g)$; in the special case when $p,q$ are mutually orthogonal elements of $\im(\hh)$, this is the same as $g=\frac{pq}{|pq|}=-h$;
\item $\Lambda\in\lin^\hh_{\bar\vartheta_p}(X,Y),\Lambda'\in\lin^\hh_{\bar\vartheta_q}(X,Y) \Rightarrow \Lambda,\Lambda',\Lambda+\Lambda'\in\lin^\cc_{gh}(X,Y)$ with $g=\im\left(\frac{q^{-1}p}{|q^{-1}p|}\right)$, $h=\bar\vartheta_p(g)=\bar\vartheta_q(g)$;  in the special case when $p,q$ are mutually orthogonal elements of $\im(\hh)$, this is the same as $g=\frac{pq}{|pq|}=h$;
\item if $g'\neq\pm g$ and $\Gamma\in\lin^\cc_{gh}(X,X')\cap\lin^\cc_{g'h'}(X,X')$, then $\Gamma\in\lin^\hh_\vartheta(X,X')$ where $\vartheta\in\aut(\hh)$ maps each element of the basis $(1,g,g',g\times g')$ into the corresponding element of the basis $(1,h,h',h\times h')$; in the special case when $h=-g,h'=-g'$, this is the same as $\vartheta=\vartheta_{g\times g'}$;
\item if $g'\neq\pm g$ and $\Lambda\in\lin^\cc_{gh}(X,Y)\cap\lin^\cc_{g'h'}(X,Y)$, then $\Lambda\in\lin^\hh_\varphi(X,Y)$ where $\varphi\in\antiaut(\hh)$ maps each element of the basis $(1,g,g',g\times g')$ into the corresponding element of the basis $(1,h,h',-h\times h')$; in the special case when $h=g,h'=g'$, this is the same as $\varphi=\bar\vartheta_{g\times g'}$.
\end{itemize}
For the bilateral quaternionic space $\hh$, the following facts are established in~\cite[Lemma 3.4]{tarallo}:
\begin{itemize}
\item $\vartheta\in\aut(\hh)\Rightarrow\lin^\hh_\vartheta(X,\hh)$ is a subspace of the right quaternionic space $\lin^\rr(X,\hh)$;
\item $\varphi\in\antiaut(\hh)\Rightarrow\lin^\hh_\varphi(X,\hh)$ is a subspace of the left quaternionic space $\lin^\rr(X,\hh)$;
\item $\vartheta\in\aut(\hh)\Rightarrow\lin^\hh_\vartheta(Y,\hh)$ is a subspace of the left quaternionic space $\lin^\rr(Y,\hh)$;
\item $\varphi\in\antiaut(\hh)\Rightarrow\lin^\hh_\varphi(Y,\hh)$ is a subspace of the right quaternionic space $\lin^\rr(Y,\hh)$.
\end{itemize}
This is because: equality \eqref{eq:LL-linear} is preserved when $\Gamma$ is replaced by $\Gamma p$ with $p\in\hh$; the analogous equality for $\Lambda:X\to\hh_r$ is preserved when $\Lambda$ is replaced by $p\Lambda$ with $p\in\hh$; the analogous equality for $\Gamma:Y\to\hh_r$ is preserved when $\Gamma$ is replaced by $p\Gamma$ with $p\in\hh$; and the analogous equality for $\Lambda:Y\to\hh_l$ is preserved when $\Lambda$ is replaced by $\Lambda p$ with $p\in\hh$. In the special case when also $X$ (or $Y$, respectively) is $\hh$, then $\lin^\rr(\hh)$ is a bilateral quaternionic space and the reference~\cite[\S3]{tarallo} also proves the following facts:
\begin{itemize}
\item $\vartheta\in\aut(\hh)\Rightarrow\lin^\hh_\vartheta(\hh_l)=\{\vartheta\,\gamma:\gamma\in\hh\}$ and $\lin^\hh_\vartheta(\hh_r)=\{\gamma\,\vartheta:\gamma\in\hh\}$ are, respectively, a right quaternionic subspace and a left quaternionic subspace of $\lin^\rr(\hh)$;
\item $\varphi\in\antiaut(\hh)\Rightarrow\lin^\hh_\varphi(\hh_l,\hh_r)=\{\lambda\,\varphi:\lambda\in\hh\}$ and $\lin^\hh_\varphi(\hh_r,\hh_l)=\{\varphi\,\lambda:\lambda\in\hh\}$ are, respectively, a left quaternionic subspace and a right quaternionic subspace of $\lin^\rr(\hh)$.
\end{itemize}

The \emph{standard quaternionic scalar product} on the left quaternionic space $\lin^\rr(\hh)$ is defined as follows in~\cite[\S2]{tarallo}. After fixing an arbitrary real orthornomal basis $\B=(e_0,e_1,e_2,e_3)$ of $\hh$, one sets
\[(\Gamma,\Lambda)_l:=\frac14\sum_{\ell=0}^3\Gamma(e_\ell)\overline{\Lambda(e_\ell)}\]
for all $\Gamma,\Lambda\in\lin^\rr(\hh)$. A direct computation shows that the result does not depend on the choice of the basis $\B$. Similarly, the \emph{standard quaternionic scalar product} on the right quaternionic space $\lin^\rr(\hh)$ is defined by the formula $(\Gamma,\Lambda)_r:=\frac14\sum_{\ell=0}^3\overline{\Gamma(e_\ell)}\Lambda(e_\ell)$ for all $\Gamma,\Lambda\in\lin^\rr(\hh)$. After defining the standard real scalar product of $\Gamma,\Lambda\in\lin^\rr(\hh)$ as $\langle\Gamma,\Lambda\rangle:=\frac14\sum_{\ell=0}^3\langle\Gamma(e_\ell),\Lambda(e_\ell)\rangle$, the cited~\cite[\S2]{tarallo} proves the following properties:
\begin{itemize}
\item for all $\Gamma,\Lambda\in\lin^\rr(\hh)$ and for all $p,q\in\hh$, the equalities $\re(\Gamma,\Lambda)_l=\langle\Gamma,\Lambda\rangle=\re(\Gamma,\Lambda)_r$, $(p\Gamma,q\Lambda)_l=p(\Gamma,\Lambda)_l\bar q$ and $(\Gamma p,\Lambda q)_r=\bar p(\Gamma,\Lambda)_r q$ hold true;
\item for all $p,q\in\hh^*$, the equalities $(\vartheta_p,\vartheta_q)_r=(\bar\vartheta_p,\bar\vartheta_q)_l=\re(\bar uv)u\bar v=\langle u,v\rangle u\bar v$ hold true with $u:=\frac{p}{|p|}, v:=\frac{q}{|q|}$; in particular, $(\vartheta_p,\vartheta_p)_r=(\bar\vartheta_p,\bar\vartheta_p)_l=1$ and the quaternion $(\vartheta_p,\vartheta_q)_r=(\bar\vartheta_p,\bar\vartheta_q)_l$ is $0$ if, and only if, $p \perp q$;
\item if $(\vartheta,\vartheta')_r=0$, then all $\Gamma\in\lin^\hh_\vartheta(\hh_l),\Gamma'\in\lin^\hh_{\vartheta'}(\hh_l)$ fulfill the equality $(\Gamma,\Gamma')_r=0$;
\item if $(\varphi,\varphi')_l=0$, then all $\Lambda\in\lin^\hh_\varphi(\hh_l,\hh_r),\Lambda'\in\lin^\hh_{\varphi'}(\hh_l,\hh_r)$ fulfill the equality $(\Lambda,\Lambda')_l=0$.
\end{itemize}
Overall, see~\cite[\S4]{tarallo}, every basis $\B=(e_0,e_1,e_2,e_3)$ of $\hh$ yields four decompositions of $\lin^\rr(\hh)$ into direct sums, which are orthogonal sums if the basis $\B$ is orthogonal:
\begin{itemize}
\item $\lin^\rr(\hh)=\lin^\hh_{\vartheta_{e_0}}(\hh_l)\oplus\lin^\hh_{\vartheta_{e_1}}(\hh_l)\oplus\lin^\hh_{\vartheta_{e_2}}(\hh_l)\oplus\lin^\hh_{\vartheta_{e_3}}(\hh_l)$;
\item $\lin^\rr(\hh)=\lin^\hh_{\bar\vartheta_{e_0}}(\hh_l,\hh_r)\oplus\lin^\hh_{\bar\vartheta_{e_1}}(\hh_l,\hh_r)\oplus\lin^\hh_{\bar\vartheta_{e_2}}(\hh_l,\hh_r)\oplus\lin^\hh_{\bar\vartheta_{e_3}}(\hh_l,\hh_r)$;
\item $\lin^\rr(\hh)=\lin^\hh_{\vartheta_{e_0}}(\hh_r)\oplus\lin^\hh_{\vartheta_{e_1}}(\hh_r)\oplus\lin^\hh_{\vartheta_{e_2}}(\hh_r)\oplus\lin^\hh_{\vartheta_{e_3}}(\hh_r)$;
\item $\lin^\rr(\hh)=\lin^\hh_{\bar\vartheta_{e_0}}(\hh_r,\hh_l)\oplus\lin^\hh_{\bar\vartheta_{e_1}}(\hh_r,\hh_l)\oplus\lin^\hh_{\bar\vartheta_{e_2}}(\hh_r,\hh_l)\oplus\lin^\hh_{\bar\vartheta_{e_3}}(\hh_r,\hh_l)$.
\end{itemize}
If $e_0=1$, then $\lin^\hh_{\vartheta_{e_0}}(\hh_l)=\lin^\hh_{id}(\hh_l)$ is the space of quaternionic linear self-maps of $\hh_l$, $\lin^\hh_{\bar\vartheta_{e_0}}(\hh_l,\hh_r)=\lin^\hh_{cj}(\hh_l,\hh_r)$ is the space of quaternionic antilinear maps $\hh_l\to\hh_r$ and so forth. For a single map $\Lambda\in\lin^\rr(\hh)$, we find four decompositions:
\begin{itemize}
\item there exist unique $\gamma_0,\gamma_1,\gamma_2,\gamma_3\in\hh$ such that $\Lambda=\sum_{\ell=0}^3\vartheta_{e_\ell}\gamma_\ell$; if, moreover, the basis $\B$ is orthogonal, then the summands in the last sum are mutually orthogonal and $\gamma_\ell=(\vartheta_{e_\ell},\Lambda)_r$;
\item there exist unique $\lambda_0,\lambda_1,\lambda_2,\lambda_3\in\hh$ such that $\Lambda=\sum_{\ell=0}^3\lambda_\ell\bar\vartheta_{e_\ell}$; if, moreover, the basis $\B$ is orthogonal, then the summands in the last sum are mutually orthogonal and $\lambda_\ell=(\Lambda,\bar\vartheta_{e_\ell})_l$;
\item there exist unique $\gamma_0,\gamma_1,\gamma_2,\gamma_3\in\hh$ such that $\Lambda=\sum_{\ell=0}^3\gamma_\ell\vartheta_{e_\ell}$; if, moreover, the basis $\B$ is orthogonal, then the summands in the last sum are mutually orthogonal and $\gamma_\ell=(\Lambda,\vartheta_{e_\ell})_l$;
\item there exist unique $\lambda_0,\lambda_1,\lambda_2,\lambda_3\in\hh$ such that $\Lambda=\sum_{\ell=0}^3\bar\vartheta_{e_\ell}\lambda_\ell$; if, moreover, the basis $\B$ is orthogonal, then the summands in the last sum are mutually orthogonal and $\lambda_\ell=(\bar\vartheta_{e_\ell},\Lambda)_r$.
\end{itemize}

\begin{example}\label{ex:fourtimesrealpart}
Consider the real linear map $\re:\hh\to\rr\subset\hh$. If $\B=(e_0,e_1,e_2,e_3)$ is an orthonormal basis of $\hh$ with $e_0=1$, then $\re=\sum_{\ell=0}^3\frac14\vartheta_{e_\ell}=\sum_{\ell=0}^3\frac14\bar\vartheta_{e_\ell}$ as a consequence of the following equalities:
\begin{align*}
\sum_{\ell=0}^3\vartheta_{e_\ell}(e_0) &=\sum_{\ell=0}^3 e_\ell \bar e_\ell=1+1+1+1=4\,,\\
\sum_{\ell=0}^3\vartheta_{e_\ell}(e_1) &= e_1-e_1e_1e_1-e_2e_1e_2-e_3e_1e_3=e_1+e_1-e_1-e_1=0\,,\\
\sum_{\ell=0}^3\vartheta_{e_\ell}(e_2) &= e_2-e_1e_2e_1-e_2e_2e_2-e_3e_2e_3=e_2-e_2+e_2-e_2=0\,,\\
\sum_{\ell=0}^3\vartheta_{e_\ell}(e_3) &= e_3-e_1e_3e_1-e_2e_3e_2-e_3e_3e_3=e_3-e_3-e_3+e_3=0\,.\qedhere
\end{align*}
\end{example}

\subsection{Regular real linear maps}\label{subsec:regularlinearmaps}

The work~\cite[\S4]{tarallo} singles out the subspace $\lin^\rr_{\mathcal{RF}}(\hh):=\{\Lambda\in\lin^\rr(\hh):(\Lambda,cj)_l=0\}$ and calls its elements \emph{regular} real linear self-maps of $\hh_l$. The reason for this notation and this terminology will become clear in the forthcoming Subsection~\ref{subsec:regularfunctions}. If we fix a basis $(1,e_1,e_2,e_3)$ of $\hh$ with $e_1,e_2,e_3\in\im(\hh)$, then $\lin^\rr_{\mathcal{RF}}(\hh)=\lin^\hh_{\bar\vartheta_{e_1}}(\hh_l,\hh_r)\oplus\lin^\hh_{\bar\vartheta_{e_2}}(\hh_l,\hh_r)\oplus\lin^\hh_{\bar\vartheta_{e_3}}(\hh_l,\hh_r)$. In other words, every regular real linear map $\Lambda:\hh_l\to\hh_l$ can be expressed as $\Lambda=\sum_{\ell=1}^3\lambda_\ell\bar\vartheta_{e_\ell}$ for some $\lambda_1,\lambda_2,\lambda_3\in\hh$. In the same work~\cite[\S4]{tarallo}, the quaternionic \emph{size} $s(\Lambda)$ of $\Lambda$ is defined as the minimal number of nonzero $\lambda_i$'s when $e_1,e_2,e_3$ are chosen to be mutually orthogonal elements of $\im(\hh)$.

\begin{example}\label{ex:size0}
The only $\Lambda\in\lin^\rr_{\mathcal{RF}}(\hh)$ having size $s(\Lambda)=0$ is $\Lambda\equiv0$.
\end{example}

\begin{example}\label{ex:size1}
A simple example of element of $\lin^\rr_{\mathcal{RF}}(\hh)$ having size $1$ is
\[\bar\vartheta_i(z_1+z_2j)=i(\bar z_1-j\bar z_2)(-i)=\bar z_1+j\bar z_2\,,\]
which belongs to $\lin^\hh_{\bar\vartheta_i}(\hh_l,\hh_r)$. Similarly, $\bar\vartheta_j(z_1+z_2j)=j(\bar z_1-j\bar z_2)(-j)=z_1-jz_2$ (belonging to $\lin^\hh_{\bar\vartheta_j}(\hh_l,\hh_r)$) and $\bar\vartheta_k(z_1+z_2j)=k(\bar z_1-j\bar z_2)(-k)=z_1+jz_2$ (belonging to $\lin^\hh_{\bar\vartheta_k}(\hh_l,\hh_r)$) have size $1$. In general, if $g\in\s$ then every element of $\lin^\hh_{\bar\vartheta_g}(\hh_l,\hh_r)$ is an element of $\lin^\rr_{\mathcal{RF}}(\hh)$ with size $1$.
\end{example}

\begin{example}\label{ex:size2}
A simple example of element of $\lin^\rr_{\mathcal{RF}}(\hh)$ having size $2$ is the Fueter variable $\zeta_1=x_1-ix_0$, i.e., the real linear map $\Lambda_i$ defined by the formula
\[\Lambda_i(z_1+z_2j)=-iz_1\,.\]
Indeed: $\Lambda_i$ is regular and $s(\Lambda_i)\leq2$ because $\Lambda_i=-\frac{i}2\bar\vartheta_j-\frac{i}2\bar\vartheta_k$. Moreover, $\Lambda_i$ cannot be expressed as $\lambda\,\bar\vartheta_g$ for any $g\in\im(\hh)^*,\lambda\in\hh$, because  $\Lambda_i$ is not injective. Thus, the size $s(\Lambda_i)$ is exactly $2$.

Similarly, the Fueter variables $\zeta_2=x_2-jx_0$ (i.e., the real linear map $\Lambda_j(z_1+z_2j)=\frac12(z_2+\bar z_2)-\frac{j}2(z_1+\bar z_1)$) and $\zeta_3=x_3-kx_0$ (i.e., the real linear map $\Lambda_k$ defined by the formula $\Lambda_k(z_1+z_2j)=-\frac{i}2(z_2-\bar z_2)-\frac{k}2(z_1+\bar z_1)$) are elements of $\lin^\rr_{\mathcal{RF}}(\hh)$ having size $2$.
\end{example}

\begin{example}\label{ex:size3}
The map $\Lambda(z_1+z_2j)=2z_1+\bar z_1+j\bar z_2$ belongs to $\lin^\rr_{\mathcal{RF}}(\hh)$. It has $s(\Lambda)=3$ because not only $\Lambda=\bar\vartheta_i+\bar\vartheta_j+\bar\vartheta_k$, but the decomposition of $\Lambda$ with respect to any orthogonal basis $(e_0,e_1,e_2,e_3)$ of $\hh$ having $e_0=1$ is $\Lambda=\bar\vartheta_{e_1}+\bar\vartheta_{e_2}+\bar\vartheta_{e_3}$. Indeed, $\Lambda=\sum_{\ell=0}^3\lambda_\ell\bar\vartheta_{e_\ell}$, where $\lambda_0=(\Lambda,cj)_l=0$ while, for all $\ell\in\{1,2,3\}$,
\[\lambda_\ell=(\Lambda,\bar\vartheta_{e_\ell})_l=(\bar\vartheta_i,\bar\vartheta_{e_\ell})_l+(\bar\vartheta_j,\bar\vartheta_{e_\ell})_l+(\bar\vartheta_k,\bar\vartheta_{e_\ell})_l=1\,.\]
\end{example}

The reference~\cite[\S5]{tarallo} proves the following properties, valid for all $\Lambda\in\lin^\rr(\hh)$:
\begin{itemize}
\item there exists $g\in\s$ such that $\Lambda\in\lin^\cc_{gg}(\hh_l,\hh_r)$ if, and only if, $\Lambda$ is an element of $\lin^\rr_{\mathcal{RF}}(\hh)$ of size $s(\Lambda)\leq2$;
\item if $\Lambda\in\lin^\rr_{\mathcal{RF}}(\hh)$ and there exists $g,h\in\s$ with $h\neq g$ such that $\Lambda\in\lin^\cc_{gh}(\hh_l,\hh_r)$, then $\Lambda$ has size $s(\Lambda)\leq1$.
\end{itemize}
The forthcoming Theorem~\ref{thm:classificationlinear} will provide a more precise classification in terms of size and subsume the last two properties listed.

Similarly,~\cite[\S4]{tarallo} considers the space $\lin^\rr_{\mathcal{LF}}(\hh):=\{\Lambda\in\lin^\rr(\hh):(cj,\Lambda)_r=0\}$ and calls its elements \emph{regular} real linear self-maps of $\hh_r$.  Again, if we fix a basis $(1,e_1,e_2,e_3)$ of $\hh$ with $e_1,e_2,e_3\in\im(\hh)$, then $\lin^\rr_{\mathcal{LF}}(\hh)=\lin^\hh_{\bar\vartheta_{e_1}}(\hh_r,\hh_l)\oplus\lin^\hh_{\bar\vartheta_{e_2}}(\hh_r,\hh_l)\oplus\lin^\hh_{\bar\vartheta_{e_3}}(\hh_r,\hh_l)$. In other words, every regular real linear map $\Lambda:\hh_r\to\hh_r$ can be expressed as $\Lambda=\sum_{\ell=1}^3\bar\vartheta_{e_\ell}\lambda_\ell$ for some $\lambda_1,\lambda_2,\lambda_3\in\hh$.

\subsection{Holomorphic functions and regular functions}\label{subsec:regularfunctions}

Let us consider a connected open subset $U$ of $\hh$ and a $C^1$ function $f:U\to\hh$. At each point $p\in U$, we will think of the differential $df_p:\hh\simeq T_pU\to T_{f(p)}\hh\simeq\hh$ as an element of $\lin^\rr(\hh)$.

Accordingly, we will think of a smooth orthogonal almost-complex structure $\I$ on $U$ at a single point $p\in U$, namely $\I_p$, as left multiplication times an imaginary unit $I(p)\in\s$ in $\hh_l$ and we will think of a smooth orthogonal almost-complex structure $\J$ on an open neighborhood $V$ of $f(U)$ at a single point $q\in V$ as right multiplication times an imaginary unit $J(q)\in\s$ in $\hh_r$. Thus, $f:(U,\I)\to(V,\J)$ is holomorphic if, and only if, at each $p\in U$, $df_p$ belongs to $\lin^\cc_{I(p)J(f(p))}(\hh_l,\hh_r)$. In particular, if we use the constant orthogonal complex structures $L_g(v):=gv$ and $R_h(v):=vh$ for some $g,h\in\s$, then $f:(U,L_g)\to(V,R_h)$ is holomorphic if, and only if, at each $p\in U$, $df_p$ belongs to $\lin^\cc_{gh}(\hh_l,\hh_r)$.

The work~\cite[\S4]{tarallo} called $f$ \emph{regular} if, for all $p\in U$, the real differential $df_p$ is an element of $\lin^\rr_{\mathcal{RF}}(\hh)$, that is, $(df_p,cj)_l=0$. If we fix an orthonormal basis $\B=(e_0,e_1,e_2,e_3)$ of $\hh$, then $(df_p,cj)_l=\frac14\sum_{\ell=0}^3df_p(e_\ell)e_\ell=\frac12\bar\partial_\B f(p)$, where
\[\bar\partial_\B f:=\frac12\sum_{\ell=0}^3\frac{\partial f}{\partial x_\ell}e_\ell\,.\]
By its very definition (see Subsection~\ref{subsec:regularlinearmaps}), this notion of regularity does not depend on the specific choice of $\B$. For $\B=(1,i,j,k)$, the operator $\bar\partial_\B$ is the right Cauchy-Riemann-Fueter operator $\bar\partial_r$. Thus, the notion of regularity considered here coincides with the celebrated notion of \emph{right Fueter-regular function}, recalled in Section~\ref{sec:introduction}. In particular, regular functions are harmonic and real analytic.

The work~\cite[\S4]{tarallo} also called $f$ \emph{regular} if, for all $p\in U$, the real differential $df_p$ is an element of $\lin^\rr_{\mathcal{LF}}(\hh)$, that is, $(cj,df_p)_r=0$. Now, $2(cj,df_p)_r=\frac12\sum_{\ell=0}^3e_\ell\frac{\partial f}{\partial x_\ell}(p)$ coincides, for $\B=(1,i,j,k)$, with $\bar\partial_lf(p)$, where $\bar\partial_l$ is the left Cauchy-Riemann-Fueter operator. Thus, the notion of regularity considered here coincides with the notion of \emph{left Fueter-regular function}, see Section~\ref{sec:introduction}.

If $f$ is bijective, in Section~\ref{sec:introduction}, we called $f$ RL-biregular (and $f^{-1}$ LR-biregular) if $f$ is right Fueter-regular and $f^{-1}$ is left Fueter-regular. This is equivalent to having $df_p\in\lin^\rr_{\mathcal{RF}}(\hh)$ and $(df_p)^{-1}\in\lin^\rr_{\mathcal{LF}}(\hh)$. Similar considerations apply to LL-biregular and RR-biregular functions. We will call $f$ \emph{absolutely biregular} if it is RL-biregular, LR-biregular, LL-biregular and RR-biregular.

\begin{example}
If $g\in\s$, then $\bar\vartheta_g=\bar\vartheta_g^{-1}\in\lin^\rr_{\mathcal{RF}}(\hh)\cap\lin^\rr_{\mathcal{LF}}(\hh)$. Thus, $\bar\vartheta_g$ is absolutely biregular.
\end{example}

For the sake of completeness, we point out that $C^1$ functions belonging to the kernel of the conjugate operator $\partial_r$ (or $\partial_l$, respectively), are those whose differentials are regular real linear maps $\hh_r\to\hh_l$ (or $\hh_l\to\hh_r$, respectively), in the sense defined in~\cite[\S4]{tarallo} and not treated here.

\subsection{Conformal real linear maps and functions}\label{subsec:conformal}

In this final subsection of preliminaries, let us recall a few facts about conformality in dimension four.

For $\Gamma,\Lambda\in\lin^\rr(\hh)$, the following facts are proven in~\cite[Lemma 6.2]{tarallo}.
\begin{itemize}
\item $\Gamma$ is an orientation-preserving conformal transformation of $\hh \Longleftrightarrow\Gamma\in\lin^\hh_\vartheta(\hh_l)$ for some $\vartheta\in\aut(\hh) \Longleftrightarrow \Gamma\in\lin^\hh_{\vartheta'}(\hh_r)$ for some $\vartheta'\in\aut(\hh)$. If so, then $\Gamma=\vartheta\,\gamma=\gamma\,\vartheta'$ with $\gamma:=\Gamma(1), \vartheta(x):=\Gamma(x)\gamma^{-1},\vartheta'(x):=\gamma^{-1}\Gamma(x)$ for all $x\in\hh$. Necessarily, $\vartheta'=\vartheta_{\gamma^{-1}}\circ\vartheta$.
\item $\Lambda$ is an orientation-reversing conformal transformation of $\hh \Longleftrightarrow \Lambda\in\lin^\hh_\varphi(\hh_l,\hh_r)$ for some $\varphi\in\antiaut(\hh) \Longleftrightarrow \Lambda\in\lin^\hh_{\varphi'}(\hh_r,\hh_l)$ for some $\varphi'\in\antiaut(\hh)$. If so, then $\Lambda=\lambda\,\varphi=\varphi'\lambda$ with $\lambda:=\Lambda(1),\varphi(x):=\lambda^{-1}\Lambda(x),\varphi'(x):=\Lambda(x)\lambda^{-1}$ for all $x\in\hh$. Necessarily, $\varphi'=\vartheta_\lambda\circ\varphi$.
\end{itemize}

In particular:

\begin{remark}\label{rmk:size1isconformal}
If $\Lambda\in\lin^\rr_{\mathcal{RF}}(\hh)$ has $s(\Lambda)=1$, then $\Lambda$ is an orientation-reversing and conformal real linear transformation of $\hh$.
\end{remark}

Obviously, one can construct conformal functions $\hh\to\hh$ by composing quaternionic translations with conformal real linear maps $\hh\to\hh$. Thus:
\begin{itemize}
\item all orientation-preserving and conformal real affine transformations of $\hh$ can be obtained as $\vartheta_p\,\gamma+\mu = \gamma\,\vartheta_q+\mu$ for $\gamma,p,q\in\hh^*$ (with $p=\gamma q$) and $\mu\in\hh$;
\item all orientation-reversing and conformal real affine transformations of $\hh$ can be obtained as $\lambda\,\bar\vartheta_q+\mu=\bar\vartheta_p\,\lambda+\mu$ for $\lambda,p,q\in\hh^*$ (with $\lambda q=p$) and $\mu\in\hh$.
\end{itemize}

A class of examples of conformal quaternionic functions that are not real affine is the family $\{\tau^{q_0}:q_0\in\hh\}$, where
\[\tau^{q_0}: \hh\setminus\{q_0\} \to \hh\setminus\{q_0\},\quad q\mapsto q_0+\frac{q-q_0}{|q-q_0|^2}\]
is called the \emph{inversion} with respect to $q_0$ in $\hh$. By direct inspection, $\tau^0(q)=\bar q^{-1}$ is an involutive self-map of $\hh\setminus\{0\}$ and
\[d\tau^0_q(v)=-\bar q^{-1}\bar v\bar q^{-1}=-\bar q^{-2}\bar\vartheta_{\bar q}(v)=\bar\vartheta_{(\bar q)^{-1}}(v)(-\bar q^{-2})\]
for all $q\in\hh\setminus\{0\}$ and for all $v\in T_q(\hh\setminus\{0\})\simeq\hh$. Since $\tau^{q_0}(q)=\tau^0(q-q_0)+q_0$ (whence $d\tau^{q_0}_p=d\tau^{0}_{p-q_0}$ for  all $p\in\hh\setminus\{q_0\}$), we conclude that
\begin{itemize}
\item $(\tau^{q_0})^{-1}=\tau^{q_0}$;
\item $d\tau^{q_0}_p=-(\bar p-\bar q_0)^{-2}\bar\vartheta_{\bar p-\bar q_0}$ for all $p\in\hh\setminus\{q_0\}$.
\end{itemize}
We will later need the following well-known theorem (see, e.g.,~\cite[pp. 222--226]{libroberger} or~\cite[Theorem B.1]{tarallo}).
\begin{itemize}
\item {[Liouville's Theorem]} Let $U$ be a nonempty connected open subset of $\hh$. If $f:U\to\hh$ is conformal, then it is (the restriction to $U$ of) either a real affine transformation of $\hh$ or a real affine transformation of $\hh$ precomposed with an inversion $\tau^{q_0}$ (for some $q_0\in\hh\setminus U$).
\end{itemize}


\section{New results about regularity and complex linearity of real linear maps}\label{sec:regularlinearmaps}

\subsection{General real linear maps}

In preparation for our work on right Fueter-regular functions, we need to deepen the study of regular real linear maps, see Subsection~\ref{subsec:regularlinearmaps}. We begin by further characterizing complex linear maps among real linear maps. We will repeatedly use, for $p,q\in\hh$, the notations $L_p(x):=px$ and $R_q(x):=xq$. We point out that $L_p$ (respectively, $R_q$) is: a conformal real linear transformation if, and only if, $p\neq0$ (respectively, $q\neq0$); an orthogonal real linear transformation if, and only if, $|p|=1$ (respectively, $|q|=1$). Moreover, if $g,h\in\s$ then $L_g^{-1}=-L_g, R_h^{-1}=-R_h$.

\begin{remark}\label{rmk:characterizationofcomplexlinearity}
For $\Lambda\in\lin^\rr(\hh)$ and $g,h\in\s$,
\[\Lambda\in\lin^\cc_{gh}(\hh_l,\hh_r)\Longleftrightarrow\Lambda(gx)=\Lambda(x)h\ \forall\,x\in\hh\Longleftrightarrow\Lambda\circ L_{g}=R_h\circ\Lambda\Longleftrightarrow\Lambda = -R_{h}\circ\Lambda\circ L_{g}\,.\]
\end{remark}

To better study complex linearity (and, later on, regularity), we will use the tools provided by the next definition. We recall that the standard real scalar product of $\Gamma,\Lambda\in\lin^\rr(\hh)$ has been defined as $\langle\Gamma,\Lambda\rangle:=\frac14\sum_{\ell=0}^3\langle\Gamma(e_\ell),\Lambda(e_\ell)\rangle$ for any real orthornomal basis $\B=(e_0,e_1,e_2,e_3)$ of $\hh$ and that its value is independent of the specific choice of $\B$.

\begin{definition}
To each $\Lambda\in\lin^\rr(\hh)$, we associate the bilinear forms $\A_\Lambda,\M_\Lambda:\im(\hh)\times\im(\hh)\to\rr$, defined as
\begin{align*}
\A_\Lambda(p,q)&=\langle R_p\circ\Lambda,\Lambda\circ L_q\rangle\\
\M_\Lambda(p,q)&=\frac12\tr\left(\A_\Lambda\right)\langle p,q\rangle-\frac12\A_\Lambda(p,q)\,.
\end{align*}
We define $\L_\Lambda:\im(\hh)\to\im(\hh)$ to be the real linear map such that $\A_\Lambda(p,q)=\langle p, \L_\Lambda(q)\rangle$ for all $p,q\in\im(\hh)$ and $\Q_\Lambda:\im(\hh)\to\rr$ to be the quadratic form defined by the formula $\Q_\Lambda(q)=| \L_\Lambda(q)|^2$ for $q\in\im(\hh)$.
For each orthonormal basis $\B=(e_0,e_1,e_2,e_3)$ of $\hh$ with $e_0=1$: $A=A_{\Lambda,\B}$ will denote the $3\times3$ real matrix associated to $\A_\Lambda$ with respect to $\widehat{\B}:=(e_1,e_2,e_3)$, which also represents $\L_\Lambda$ with respect to $\widehat{\B},\widehat{\B}$; while $M=M_{\Lambda,\B}$ will denote the $3\times3$ real matrix associated to $\M_\Lambda$ with respect to $\widehat{\B}$.
\end{definition}

\begin{remark}\label{rmk:quadraticform}
Let $\Lambda\in\lin^\rr(\hh)$, let $\B=(e_0,e_1,e_2,e_3)$ be an orthonormal basis of $\hh$ with $e_0=1$ and set $A:=A_{\Lambda,\B}$. Then $A^TA$ is the matrix associated to $\Q_\Lambda$ with respect to the basis $(e_1,e_2,e_3)$. Indeed, if $p=p_1e_1+p_2e_2+p_3e_3\in\s$, then
\[\Q_\Lambda(p)=| \L_\Lambda(p)|^2=\Vert AP\Vert^2=(AP)^TAP=P^TA^TA\,P,\quad P:=\begin{bmatrix} p_1\\p_2\\p_3\end{bmatrix}\,.\]
\end{remark}

The significance of $\A_\Lambda,\L_\Lambda,\Q_\Lambda$ is readily explained by the next result, which uses the notation $\Vert\Lambda\Vert:=\sqrt{\langle\Lambda,\Lambda\rangle}$.

\begin{proposition}\label{prop:holomorphy}
Let $\Lambda\in\lin^\rr(\hh)$ and assume $\Lambda\neq0$.
\begin{enumerate}
\item For any $g,h\in\s$, $\A_\Lambda(h,g)\leq\left\Vert\Lambda\right\Vert^2$ and
\[\A_\Lambda(h,g)=\left\Vert\Lambda\right\Vert^2\Longleftrightarrow\Lambda\in\lin^\cc_{gh}(\hh_l,\hh_r)\,.\]
\item For each $g\in\s$, $\Q_\Lambda(g)\leq\left\Vert\Lambda\right\Vert^4$ and
\[\Q_\Lambda(g)=\left\Vert\Lambda\right\Vert^4\Longleftrightarrow\exists\,h\in\s\mathrm{\ such\ that\ }\Lambda\in\lin^\cc_{gh}(\hh_l,\hh_r)\,.\]
Moreover, $\L_{\Lambda}(g)=\left\Vert\Lambda\right\Vert^2h$. In particular, $h$ is unique.
\item The maximum eigenvalue of $\Q_\Lambda$ equals $\left\Vert\Lambda\right\Vert^4$ if, and only if, there exist $g,h\in\s$ such that $\Lambda\in\lin^\cc_{gh}(\hh_l,\hh_r)$.
\end{enumerate}
If, instead, $\Lambda=0\in\lin^\rr(\hh)$, then for all $g,h\in\s$ we have $\A_\Lambda(h,g)=0=\Q_\Lambda(g)$ and $\Lambda\in\lin^\cc_{gh}(\hh_l,\hh_r)$.
\end{proposition}

\begin{proof}
Let $g,h\in\s$: then $L_g,R_h$ are orthogonal real linear transformations and $R_h^{-1}=-R_h$. Taking into account that $\langle\cdot,\cdot\rangle$ is invariant under orthonormal changes of bases, we have
\begin{align*}
0&\leq\Vert \Lambda+R_h\circ\Lambda\circ L_g\Vert^2=\left\Vert\Lambda\right\Vert^2+2\langle \Lambda,R_h\circ\Lambda\circ L_g\rangle+\left\Vert R_h\circ\Lambda\circ L_g\right\Vert^2\\
&=2\left\Vert\Lambda\right\Vert^2-2\langle R_h\circ\Lambda,\Lambda\circ L_g\rangle=2\left\Vert\Lambda\right\Vert^2-2\A_\Lambda(h,g)\,,
\end{align*}
whence $\A_\Lambda(h,g)\leq\left\Vert\Lambda\right\Vert^2$. Equality holds if, and only if, $\Lambda=-R_h\circ\Lambda\circ L_g$, which is in turn equivalent to $\Lambda\in\lin^\cc_{gh}(\hh_l,\hh_r)$ by Remark~\ref{rmk:characterizationofcomplexlinearity}. The first statement is therefore proven.

We now turn to the second statement. For $h:=\left|\L_{\Lambda}(g)\right|^{-1}\L_{\Lambda}(g)$, we have
\[\left\Vert\Lambda\right\Vert^2\geq\A_\Lambda(h,g)=\langle h,\L_\Lambda(g)\rangle=\left|\L_{\Lambda}(g)\right|^{-1}\langle \L_\Lambda(g),\L_\Lambda(g)\rangle=\left|\L_{\Lambda}(g)\right|\,,\]
whence the desired inequality $\left\Vert\Lambda\right\Vert^4\geq\left|\L_{\Lambda}(g)\right|^2=\Q_\Lambda(g)$ follows. In both of the former inequalities, if equality holds, then $\Lambda\in\lin^\cc_{gh}(\hh_l,\hh_r)$. Conversely, let us assume $g,h\in\s$ and $\Lambda\in\lin^\cc_{gh}(\hh_l,\hh_r)$, whence $\A_\Lambda(h,g)=\left\Vert\Lambda\right\Vert^2$ and $(h,g)$ is a maximum point for $(\A_\Lambda)_{|_{\s\times\s}}$. Fix an orthonormal basis $\B=(1,e_1,e_2,e_3)$ of $\hh$, set $A:=A_{\Lambda,\B}$ and assume $g=g_1e_1+g_2e_2+g_3e_3,h=h_1e_1+h_2e_2+h_3e_3$. The gradient of the map $\im(\hh)\times\im(\hh)\times\rr^2\to\rr,\ (X,Y,\lambda,\mu)\mapsto X^TAY-\lambda(\Vert X\Vert^2-1)-\mu(\Vert Y\Vert^2-1)$ must vanish when
\[(X,Y)=(H,G),\quad G:=\begin{bmatrix} g_1\\g_2\\g_3\end{bmatrix},\quad H:=\begin{bmatrix}h_1\\h_2\\h_3\end{bmatrix}\,.\]
In particular, $AY-2\lambda X$ vanishes at $(X,Y)=(H,G)$. Thus, $AG=2\lambda H$, i.e, $\L_\Lambda(g)=2\lambda h$. It follows that
\[\left\Vert\Lambda\right\Vert^2=\A_\Lambda(h,g)=\langle h,\L_\Lambda(g)\rangle=\langle h,2\lambda h\rangle=2\lambda\]
and that
\[\Q_\Lambda(g)=\left|\L_{\Lambda}(g)\right|^2=|2\lambda|^2=\left\Vert\Lambda\right\Vert^4\,,\]
as desired. In particular, $\L_\Lambda(g)=\left\Vert\Lambda\right\Vert^2h$. When $\Lambda\neq0$, there is a unique $h\in\s$ fulfilling this equality. The second statement is therefore proven.

We conclude by proving the third statement. If the maximum eigenvalue of $\Q_\Lambda$ equals $\left\Vert\Lambda\right\Vert^4$ and if $p\in\im(\hh)$ is an eigenvector with eigenvalue $\left\Vert\Lambda\right\Vert^4$, then $g:=\frac{p}{|p|}$ is an element of $\s$ such that $\Q_\Lambda(g)=\left\Vert\Lambda\right\Vert^4$. By the second statement, there exists $h\in\s$ such that $\Lambda\in\lin^\cc_{gh}(\hh_l,\hh_r)$. Conversely, if $\Lambda\in\lin^\cc_{gh}(\hh_l,\hh_r)$ for some $g,h\in\s$, then the second statement yields that $\Q_\Lambda(g)=\left\Vert\Lambda\right\Vert^4$, whence $\left\Vert\Lambda\right\Vert^4$ is the maximum eigenvalue of $\Q_\Lambda$.
\end{proof}

The significance of $\M_\Lambda$ will become clear in the forthcoming Subsection~\ref{subsec:regulalinearmaps}. We now aim at computing the matrix $A_{\Lambda,\B}$ associated to $\A_\Lambda,\L_\Lambda$ and the matrix $M_{\Lambda,\B}$ associated to $\M_\Lambda$ with respect to a positively oriented orthonormal basis $\B$. The following technical lemma will be useful.

\begin{lemma}\label{lem:realpartofproduct}
Fix an orthonormal basis $(e_0,e_1,e_2,e_3)$ of $\hh$ with $e_0=1$. Let $\alpha,\beta\in\{1,2,3\}$ and $m,n\in\{0,1,2,3\}$. The expression $\re(\bar e_m e_\alpha e_n e_\beta)$ is different from zero only in the following cases:
\begin{itemize}
\item $\alpha=\beta$ and $m=n$
\item $\alpha\neq\beta$ and $\{m,n\}$ equals either $\{\alpha,\beta\}$ or $\{0,1,2,3\}\setminus\{\alpha,\beta\}$.
\end{itemize}
\end{lemma}

\begin{proof}
First assume $\alpha=\beta$: if $m=n$, then $\bar e_m e_\alpha e_n e_\beta=\pm e_\alpha^2=\mp1$; if $m\neq n$, then $\bar e_m e_\alpha e_n e_\beta=\bar e_m e_\alpha e_n e_\alpha=\pm e_m e_n$ is an imaginary unit and $\re(\bar e_m e_\alpha e_n e_\beta)=0$.

Now assume $\alpha\neq\beta$. If $\{m,n,\alpha,\beta\}=\{0,1,2,3\}$, then $\bar e_m e_\alpha e_n e_\beta=\pm e_1e_2e_3$ is either $1$ or $-1$. If $\{m,n,\alpha,\beta\}=\{0,\alpha,\beta\}$, then $\bar e_m e_\alpha e_n e_\beta$ equals $e_\alpha e_\beta$, or $\pm e_\alpha^2e_\beta=\mp e_\beta$, or $\pm e_\alpha e_\beta^2=\mp e_\alpha$: in all three cases, $\re(\bar e_m e_\alpha e_n e_\beta)=0$. If $\{m,n,\alpha,\beta\}=\{1,2,3\}$, then $\bar e_m e_\alpha e_n e_\beta$ equals $\pm e_1 e_2$, or $\pm e_1 e_3$, or $\pm e_2 e_3$: in all three cases, $\re(\bar e_m e_\alpha e_n e_\beta)=0$. Finally, suppose $\{m,n,\alpha,\beta\}=\{\alpha,\beta\}$: if $m=n$, then $\bar e_m e_\alpha e_n e_\beta=\pm e_\alpha e_\beta$ has $\re(\bar e_m e_\alpha e_n e_\beta)=0$; if $m\neq n$ (i.e., $\{m,n\}=\{\alpha,\beta\}$), then $\bar e_m e_\alpha e_n e_\beta=\pm e_\alpha^2 e_\beta^2$ is either $1$ or $-1$.
\end{proof}

We are now ready for the announced computation.

\begin{theorem}\label{thm:formmatrix}
Let $\Lambda\in\lin^\rr(\hh)$. Fix an orthonormal basis $\B=(e_0,e_1,e_2,e_3)$ of $\hh$ with $e_0=1$ and with $e_3=e_1 e_2$. Take $\lambda_0,\lambda_1,\lambda_2,\lambda_3\in\hh$ such that $\Lambda=\sum_{m=0}^3\lambda_m\bar\vartheta_{e_m}$. Then
\begin{align}\label{eq:matrixA}
&A_{\Lambda,\B}=\\
&\begin{bmatrix}
-|\lambda_0|^2-|\lambda_1|^2+|\lambda_2|^2+|\lambda_3|^2&-2\langle \lambda_1 e_1,\lambda_2 e_2\rangle+2\langle \lambda_0,\lambda_3 e_3\rangle&-2\langle \lambda_1 e_1, \lambda_3 e_3\rangle-2\langle \lambda_0,\lambda_2 e_2\rangle\\
-2\langle \lambda_1 e_1,\lambda_2 e_2\rangle-2\langle \lambda_0,\lambda_3 e_3\rangle&-|\lambda_0|^2+|\lambda_1|^2-|\lambda_2|^2+|\lambda_3|^2&-2\langle \lambda_2 e_2, \lambda_3 e_3\rangle+2\langle \lambda_0,\lambda_1 e_1\rangle\\
-2\langle \lambda_1 e_1, \lambda_3 e_3\rangle+2\langle\lambda_0,\lambda_2 e_2\rangle&-2\langle \lambda_2 e_2, \lambda_3 e_3\rangle-2\langle \lambda_0,\lambda_1 e_1\rangle&-|\lambda_0|^2+|\lambda_1|^2+|\lambda_2|^2-|\lambda_3|^2\\
\end{bmatrix}\,.\notag
\end{align}
In particular, $\tr(\A_\Lambda)=\tr(\M_\Lambda)=-3|\lambda_0|^2+|\lambda_1|^2+|\lambda_2|^2+|\lambda_3|^2=\|\Lambda\|^2-4|\lambda_0|^2$. Moreover,
\begin{align}\label{eq:matrixM}
&M_{\Lambda,\B}=\begin{bmatrix}
|\lambda_1|^2-|\lambda_0|^2&\langle \lambda_1 e_1,\lambda_2 e_2\rangle-\langle \lambda_0,\lambda_3 e_3\rangle&\langle \lambda_1 e_1, \lambda_3 e_3\rangle+\langle \lambda_0,\lambda_2 e_2\rangle\\
\langle \lambda_1 e_1,\lambda_2 e_2\rangle+\langle \lambda_0,\lambda_3 e_3\rangle&|\lambda_2|^2-|\lambda_0|^2&\langle \lambda_2 e_2, \lambda_3 e_3\rangle-\langle \lambda_0,\lambda_1 e_1\rangle\\
\langle \lambda_1 e_1, \lambda_3 e_3\rangle-\langle\lambda_0,\lambda_2 e_2\rangle&\langle \lambda_2 e_2, \lambda_3 e_3\rangle+\langle \lambda_0,\lambda_1 e_1\rangle&|\lambda_3|^2-|\lambda_0|^2\\
\end{bmatrix}\,.
\end{align}
\end{theorem}

\begin{proof}
For $\alpha,\beta\in\{1,2,3\}$, let us compute the coefficients $a_{\alpha\beta}:=\langle R_{e_\alpha}\circ\Lambda,\Lambda\circ L_{e_\beta}\rangle$ of the matrix $A=A_{\Lambda,\B}$:
\begin{align*}
a_{\alpha\beta}&=\langle R_{e_\alpha}\circ\Lambda,\Lambda\circ L_{e_\beta}\rangle=\re\left(R_{e_\alpha}\circ\Lambda,\Lambda\circ L_{e_\beta}\right)_l\\
&=\frac14\sum_{\ell=0}^3\re\left(R_{e_\alpha}\circ\Lambda(e_\ell)\,\overline{\Lambda\circ L_{e_\beta}(e_\ell)}\right)=\frac14\sum_{\ell=0}^3\re\left(\Lambda(e_\ell)e_\alpha\overline{\Lambda(e_\beta e_\ell)}\right)\\
&=\frac14\sum_{\ell=0}^3\sum_{m,n=0}^3\re\left(\lambda_m\bar\vartheta_{e_m}(e_\ell)e_\alpha\vartheta_{e_n}(e_\beta e_\ell)\bar\lambda_n\right)\\
&=\frac14\sum_{m,n=0}^3\re\left(\lambda_me_m\left(\sum_{\ell=0}^3\bar e_\ell\left(\bar e_m e_\alpha e_n e_\beta\right)e_\ell\right) \bar e_n\bar\lambda_n\right)\,.
\end{align*}
According to Example~\ref{ex:fourtimesrealpart}, every $v\in\hh$ has $\frac14\sum_{\ell=0}^3\bar e_\ell v e_\ell=\sum_{\ell=0}^3\frac14\vartheta_{\bar e_\ell}(v)=\re(v)$. Thus,
\begin{align}
a_{\alpha\beta}
&=\sum_{m,n=0}^3\re\left(\lambda_me_m\re\left(\bar e_m e_\alpha e_n e_\beta\right)\bar e_n\bar\lambda_n\right)\notag
\\
&=\sum_{m,n=0}^3\re(\lambda_me_m\bar e_n\bar\lambda_n)\re\left(\bar e_m e_\alpha e_n e_\beta\right)\notag\\
&=\sum_{m,n=0}^3\langle\lambda_me_m,\lambda_n e_n\rangle\re\left(\bar e_m e_\alpha e_n e_\beta\right)
\,.\label{eq:aalphabeta}
\end{align}
We now apply Lemma~\ref{lem:realpartofproduct}. For $\alpha=\beta$, formula~\eqref{eq:aalphabeta} yields
\[a_{\alpha\alpha}=\sum_{m=0}^3\langle\lambda_me_m,\lambda_m e_m\rangle\re\left(\bar e_m e_\alpha e_m e_\alpha\right)=\sum_{m=0}^3|\lambda_m|^2\re\left(\bar e_m e_\alpha e_m e_\alpha\right)\,,\]
whence
\begin{align*}
a_{11}&=-|\lambda_0|^2-|\lambda_1|^2+|\lambda_2|^2+|\lambda_3|^2\,,\\
a_{22}&=-|\lambda_0|^2+|\lambda_1|^2-|\lambda_2|^2+|\lambda_3|^2\,,\\
a_{33}&=-|\lambda_0|^2+|\lambda_1|^2+|\lambda_2|^2-|\lambda_3|^2\,.
\end{align*}
If $\alpha\neq\beta$ and if $\{\alpha,\beta,\gamma\}=\{1,2,3\}$, formula~\eqref{eq:aalphabeta} and Lemma~\ref{lem:realpartofproduct} yield
\begin{align*}
a_{\alpha\beta}&=\sum_{\{m,n\}=\{\alpha,\beta\}}\langle\lambda_me_m,\lambda_n e_n\rangle\re\left(\bar e_m e_\alpha e_n e_\beta\right)+\sum_{\{m,n\}=\{0,\gamma\}}\langle\lambda_me_m,\lambda_n e_n\rangle\re\left(\bar e_m e_\alpha e_n e_\beta\right)\\
&=\langle \lambda_\alpha e_\alpha, \lambda_\beta e_\beta\rangle\re(\bar e_\alpha e_\alpha e_\beta e_\beta)+\langle \lambda_\beta e_\beta,\lambda_\alpha e_\alpha\rangle\re(\bar e_\beta e_\alpha e_\alpha e_\beta)+\langle\lambda_0,\lambda_\gamma e_\gamma\rangle\re\left(e_\alpha e_\gamma e_\beta\right)\\
&\phantom{=}+\langle\lambda_\gamma e_\gamma,\lambda_0\rangle\re\left(\bar e_\gamma e_\alpha e_\beta\right)=-2\langle \lambda_\alpha e_\alpha, \lambda_\beta e_\beta\rangle+2\sigma_{\alpha,\beta,\gamma}\langle\lambda_0,\lambda_\gamma e_\gamma\rangle\,,
\end{align*}
where $\sigma_{\alpha,\beta,\gamma}$ is the sign of the permutation $(\alpha,\beta,\gamma)$ of $\{1,2,3\}$. In other words,
\begin{align*}
a_{12}&=-2\langle \lambda_1 e_1, \lambda_2 e_2\rangle+2\langle\lambda_0,\lambda_3 e_3\rangle\,,\\
a_{21}&=-2\langle \lambda_1 e_1, \lambda_2 e_2\rangle-2\langle\lambda_0,\lambda_3 e_3\rangle\,,\\
a_{13}&=-2\langle \lambda_1 e_1, \lambda_3 e_3\rangle-2\langle\lambda_0,\lambda_2 e_2\rangle\,,\\
a_{31}&=-2\langle \lambda_1 e_1, \lambda_3 e_3\rangle+2\langle\lambda_0,\lambda_2 e_2\rangle\,,\\
a_{23}&=-2\langle \lambda_2 e_2, \lambda_3 e_3\rangle+2\langle\lambda_0,\lambda_1 e_1\rangle\,,\\
a_{32}&=-2\langle \lambda_2 e_2, \lambda_3 e_3\rangle-2\langle\lambda_0,\lambda_1 e_1\rangle\,.
\end{align*}
We have thus proven formula~\eqref{eq:matrixA}. The equality
\[\tr(A)=-3|\lambda_0|^2+|\lambda_1|^2+|\lambda_2|^2+|\lambda_3|^2=\|\Lambda\|^2-4|\lambda_0|^2\]
easily follows. If we denote the identity in $GL(3,\rr)$ as $I_3$, then the definition of $\M_\Lambda$ yields
\[M_{\Lambda,\B}=\frac12\left(\tr(A)I_3-A\right)\]
whence both formula~\eqref{eq:matrixM} and the equality $\tr\left(M_{\Lambda,\B}\right)=\frac12\left(3\tr(A)-\tr(A)\right)=\tr(A)$ follow.
\end{proof}

\begin{example}
The real linear map $cj=\bar\vartheta_1$ has $\Vert cj\Vert^2=1$ and $A_{cj,\B}=-I_3$ with respect to the standard basis $\B$ of $\hh$. It follows that $\A_{cj}(p,q)=-\langle p,q\rangle,\L_{cj}(q)=-q$ and $\Q_{cj}(q)=\Vert q\Vert^2$ for all $p,q\in\im(\hh)$. For all $g\in\s$, we have $\Q_{cj}(g)=1$. Moreover,
\[cj\in\lin_{gh}^\cc(\hh_l,\hh_r)\Longleftrightarrow h=\L_{cj}(g)=-g\,.\]
\end{example}

\begin{example}
The real linear map $\Lambda(z_1+z_2j)=2z_1+\bar z_1+z_2-2\bar z_2$ can be expressed as
\[\Lambda=\frac{1-2j}2\bar\vartheta_1+\frac{1+2j}2\bar\vartheta_i+\frac{2+j}2\bar\vartheta_j+\frac{2-j}2\bar\vartheta_k\]
with respect to the standard basis $\B$ of $\hh$. Thus, $|\lambda_\ell|^2=\frac54$ for all $\ell\in\{0,1,2,3\}$; moreover, $\lambda_1=\bar\lambda_0,\lambda_2j=-\lambda_0$ and $\lambda_3 k=-\lambda_1i$, whence $\langle \lambda_0,\lambda_2 j\rangle,\langle \lambda_1 i,\lambda_3 k\rangle$ both equal $-|\lambda_0|^2=-\frac54$, while $\langle \lambda_0,\lambda_1 i\rangle,-\langle \lambda_0,\lambda_3 k\rangle,-\langle \lambda_1 i,\lambda_2 j\rangle,\langle \lambda_2 j,\lambda_3 k\rangle$ all equal $\langle \lambda_0,\bar\lambda_0i\rangle=0$. By direct computation, $\Vert \Lambda\Vert^2=4\frac54=5$,
\[A=A_{\Lambda,\B}=
\begin{bmatrix}
0&0&5\\
0&0&0\\
0&0&0
\end{bmatrix},\quad
A^TA=
\begin{bmatrix}
0&0&0\\
0&0&0\\
0&0&25
\end{bmatrix}\,.\]
It follows that the only $g\in\s$ such that $\Q_{\Lambda}(g)=\Vert \Lambda\Vert^4=25$ are $g=\pm k$. Moreover, $\L_\Lambda(\pm k)=\pm 5i$, whence
\[\Lambda\in\lin_{gh}^\cc(\hh_l,\hh_r)\Longleftrightarrow (g,h)\in\{(k,i),(-k,-i)\}\,.\]
\end{example}

For future reference, we also consider the case when the basis $\B$ is not positively oriented. Going through the details of the proof of Theorem~\ref{thm:formmatrix}, we can make the following remark.

\begin{remark}\label{rmk:negativelyoriented}
Let $\Lambda\in\lin^\rr(\hh)$. Fix an orthonormal basis $\B=(e_0,e_1,e_2,e_3)$ of $\hh$ with $e_0=1$ and with $e_3=e_2 e_1$. If $\lambda_0,\lambda_1,\lambda_2,\lambda_3\in\hh$ are such that $\Lambda=\sum_{m=0}^3\lambda_m\bar\vartheta_{e_m}$, then
\begin{align*}
&A_{\Lambda,\B}=\\
&\begin{bmatrix}
-|\lambda_0|^2-|\lambda_1|^2+|\lambda_2|^2+|\lambda_3|^2&-2\langle \lambda_1 e_1,\lambda_2 e_2\rangle-2\langle \lambda_0,\lambda_3 e_3\rangle&-2\langle \lambda_1 e_1, \lambda_3 e_3\rangle+2\langle \lambda_0,\lambda_2 e_2\rangle\\
-2\langle \lambda_1 e_1,\lambda_2 e_2\rangle+2\langle \lambda_0,\lambda_3 e_3\rangle&-|\lambda_0|^2+|\lambda_1|^2-|\lambda_2|^2+|\lambda_3|^2&-2\langle \lambda_2 e_2, \lambda_3 e_3\rangle-2\langle \lambda_0,\lambda_1 e_1\rangle\\
-2\langle \lambda_1 e_1, \lambda_3 e_3\rangle-2\langle\lambda_0,\lambda_2 e_2\rangle&-2\langle \lambda_2 e_2, \lambda_3 e_3\rangle+2\langle \lambda_0,\lambda_1 e_1\rangle&-|\lambda_0|^2+|\lambda_1|^2+|\lambda_2|^2-|\lambda_3|^2\\
\end{bmatrix}\,.
\end{align*}
and
\begin{align*}
&M_{\Lambda,\B}=\begin{bmatrix}
|\lambda_1|^2-|\lambda_0|^2&\langle \lambda_1 e_1,\lambda_2 e_2\rangle+\langle \lambda_0,\lambda_3 e_3\rangle&\langle \lambda_1 e_1, \lambda_3 e_3\rangle-\langle \lambda_0,\lambda_2 e_2\rangle\\
\langle \lambda_1 e_1,\lambda_2 e_2\rangle-\langle \lambda_0,\lambda_3 e_3\rangle&|\lambda_2|^2-|\lambda_0|^2&\langle \lambda_2 e_2, \lambda_3 e_3\rangle+\langle \lambda_0,\lambda_1 e_1\rangle\\
\langle \lambda_1 e_1, \lambda_3 e_3\rangle+\langle\lambda_0,\lambda_2 e_2\rangle&\langle \lambda_2 e_2, \lambda_3 e_3\rangle-\langle \lambda_0,\lambda_1 e_1\rangle&|\lambda_3|^2-|\lambda_0|^2\\
\end{bmatrix}\,.
\end{align*}
\end{remark}

\subsection{Regular real linear maps}\label{subsec:regulalinearmaps}

We now turn back to the study of the left quaternionic subspace $\lin^\rr_{\mathcal{RF}}(\hh)$ of $\lin^\rr(\hh)$. With respect to an orthonormal basis $\B=(e_0,e_1,e_2,e_3)$ of $\hh$ with $e_0=1$, a real linear map $\Lambda=\sum_{m=0}^3\lambda_m\bar\vartheta_{e_m}$ is regular if, and only if, $\lambda_0=0$. Therefore, Theorem~\ref{thm:formmatrix} has the following immediate consequence.

\begin{corollary}\label{cor:matrixcharacterization}
Let $\Lambda\in\lin^\rr(\hh)$. The inclusion $\Lambda\in\lin^\rr_{\mathcal{RF}}(\hh)$ holds if, and only if, the number $\tr (\A_{\Lambda})=\tr (\M_{\Lambda})$ equals $\|\Lambda\|^2$.
\end{corollary}

We can state yet another consequence of Theorem~\ref{thm:formmatrix} and of Remark~\ref{rmk:negativelyoriented}.

\begin{corollary}\label{cor:formmatrix}
If $\Lambda\in\lin^\rr_{\mathcal{RF}}(\hh)$ decomposes as $\Lambda=\lambda_1\bar\vartheta_{e_1}+\lambda_2\bar\vartheta_{e_2}+\lambda_3\bar\vartheta_{e_3}$ with respect to any orthonormal basis $\B=(1,e_1,e_2,e_3)$ of $\hh$, then
\begin{align*}
A_{\Lambda,\B}&=\begin{bmatrix}
-|\lambda_1|^2+|\lambda_2|^2+|\lambda_3|^2&-2\langle \lambda_1 e_1,\lambda_2 e_2\rangle&-2\langle \lambda_1 e_1, \lambda_3 e_3\rangle\\
-2\langle \lambda_1 e_1,\lambda_2 e_2\rangle&|\lambda_1|^2-|\lambda_2|^2+|\lambda_3|^2&-2\langle \lambda_2 e_2, \lambda_3 e_3\rangle\\
-2\langle \lambda_1 e_1, \lambda_3 e_3\rangle&-2\langle \lambda_2 e_2, \lambda_3 e_3\rangle&|\lambda_1|^2+|\lambda_2|^2-|\lambda_3|^2\\
\end{bmatrix}\,,\\
M_{\Lambda,\B}&=\begin{bmatrix}
|\lambda_1|^2&\langle \lambda_1 e_1, \lambda_2 e_2\rangle&\langle \lambda_1 e_1, \lambda_3 e_3\rangle\\
\langle \lambda_1 e_1, \lambda_2 e_2\rangle&|\lambda_2|^2&\langle \lambda_2 e_2, \lambda_3 e_3\rangle\\
\langle \lambda_1 e_1, \lambda_3 e_3\rangle&\langle \lambda_2 e_2, \lambda_3 e_3\rangle&|\lambda_3|^2\\
\end{bmatrix}\,.
\end{align*}
In other words, $M_{\Lambda,\B}$ is the Gram matrix of $\lambda_1 e_1,\lambda_2 e_2,\lambda_3 e_3$. In particular, $\A_\Lambda,\M_\Lambda$ are both symmetric and $\M_\Lambda$ is positive semidefinite. If, moreover, $\B$ is chosen to make $A_{\Lambda,\B}$ diagonal, then $A_{\Lambda,\B}=\mathrm{diag}(-|\lambda_1|^2+|\lambda_2|^2+|\lambda_3|^2,|\lambda_1|^2-|\lambda_2|^2+|\lambda_3|^2,|\lambda_1|^2+|\lambda_2|^2-|\lambda_3|^2)$, $M_{\Lambda,\B}=\mathrm{diag}(|\lambda_1|^2,|\lambda_2|^2,|\lambda_3|^2)$ and the quaternions $\{\lambda_1 e_1,\lambda_2 e_2,\lambda_3e_3\}$ are mutually orthogonal.
\end{corollary}

We can now study the size $s(\Lambda)$ of $\Lambda\in\lin^\rr_{\mathcal{RF}}(\hh)$, defined in Subsection~\ref{subsec:regularlinearmaps} as the minimum number of nonzero $\lambda_\ell$'s when $\Lambda$ is decomposed as $\Lambda=\lambda_1\bar\vartheta_{e_1}+\lambda_2\bar\vartheta_{e_2}+\lambda_3\bar\vartheta_{e_3}$ with respect to an orthonormal basis $\B=(1,e_1,e_2,e_3)$ of $\hh$.

\begin{remark}\label{rmk:size}
If $\Lambda\in\lin^\rr_{\mathcal{RF}}(\hh)$, then $s(\Lambda)=\rank(\M_\Lambda)$. Indeed, $s(\Lambda)\leq\rank(\M_\Lambda)$ because $\B$ can be chosen so that $M_{\Lambda,\B}=\mathrm{diag}(|\lambda_1|^2,|\lambda_2|^2,|\lambda_3|^2)$. On the other hand, $s(\Lambda)\geq\rank(\M_\Lambda)$ because, by Corollary~\ref{cor:formmatrix}, the $\ell$th row and the $\ell$th column of $M_{\Lambda,\B}$ vanish when $\lambda_\ell=0$.

\end{remark}

We now wish to study which real linear maps belong to $\lin^\cc_{gg}(\hh_l,\hh_r)$. The next characterization of regularity will be useful.

\begin{lemma}\label{lem:characterizationofregularity}
Let us fix an orthonormal basis $(e_0,e_1,e_2,e_3)$ of $\hh$, with $e_0=1$. For each $\Lambda\in\lin^\rr(\hh)$, the following properties are equivalent:
\begin{enumerate}
\item $\Lambda\in\lin^\rr_{\mathcal{RF}}(\hh)$;
\item $\Lambda = - \sum_{\ell=1}^3 R_{e_\ell}\circ\Lambda\circ L_{e_\ell}$
\item $\sum_{\ell=0}^3 R_{e_\ell}\circ\Lambda\circ L_{e_\ell}=0$;
\item $\sum_{\ell=0}^3 \Lambda(e_\ell)e_\ell=0$.
\end{enumerate}
\end{lemma}

\begin{proof}
Since $R_{e_0}\circ\Lambda\circ L_{e_0}=\Lambda$, it is easy to see that property {\it 2.} implies property {\it 3}. Property {\it 3} implies $0=\sum_{\ell=0}^3 R_{e_\ell}\circ\Lambda\circ L_{e_\ell}(1)=\sum_{\ell=0}^3 \Lambda(e_\ell)e_\ell$, whence property {\it 4} follows. Property {\it 4} is equivalent to regularity, i.e., property {\it 1}, because
\[(\Lambda,cj)_l=\frac14\sum_{\ell=0}^3\Lambda(e_\ell)\overline{cj(e_\ell)}=\frac14\sum_{\ell=0}^3 \Lambda(e_\ell)e_\ell.\]

To complete the proof, we assume that property {\it 4} holds and prove property {\it 2} or, equivalently, that
\[\Lambda(e_\alpha) = - \sum_{\ell=1}^3 R_{e_\ell}\circ\Lambda\circ L_{e_\ell}(e_\alpha)\]
for all $\alpha\in\{0,1,2,3\}$. The last equality is equivalent to
\[\Lambda(e_\alpha) = - \sum_{\ell=1}^3 \Lambda(e_\ell e_\alpha)e_\ell=-\sum_{m\in\{0,1,2,3\}\setminus\{\alpha\}}\Lambda(e_m)e_m\bar e_\alpha\]
and to
\[\Lambda(e_\alpha)e_\alpha=-\sum_{m\in\{0,1,2,3\}\setminus\{\alpha\}}\Lambda(e_m)e_m\,,\]
which is true by property {\it 4}.
\end{proof}

We are now in a position to study which real linear maps belong to $\lin^\cc_{gg}(\hh_l,\hh_r)$.

\begin{theorem}\label{thm:lineargg}
Let $\Lambda\in\lin^\rr(\hh)$.
\begin{enumerate}
\item If $\Lambda\in\lin^\cc_{gg}(\hh_l,\hh_r)$ for some $g\in\s$, then $\Lambda\in\lin^\rr_{\mathcal{RF}}(\hh)$.
\end{enumerate}
Now assume $\Lambda\in\lin^\rr_{\mathcal{RF}}(\hh)$, denote by $c(\Lambda)$ the maximum number of linearly independent elements $g\in\s$ such that $\Lambda\in\lin^\cc_{gg}(\hh_l,\hh_r)$ and fix an orthonormal basis $\B=(1,e_1,e_2,e_3)$ of $\hh$. Then:
\begin{enumerate}
\item[2.] If $g=g_1e_1+g_2e_2+g_3e_3\in\s$, then
\[\Lambda\in\lin^\cc_{gg}(\hh_l,\hh_r)\Longleftrightarrow\M_\Lambda(g,g)=0\Longleftrightarrow\begin{bmatrix} g_1\\g_2\\g_3\end{bmatrix}\in\ker(M_{\Lambda,\B})\,.\]
\item[3.] $c(\Lambda)=\dim\ker(M_{\Lambda,\B})$.
\item[4.] $c(\Lambda)=3-s(\Lambda)$.
\end{enumerate}
\end{theorem}

\begin{proof}
We first establish property {\it 1}. If $g\in\s$ is such that $\Lambda\in\lin^\cc_{gg}(\hh_l,\hh_r)$ and if $h\in\s$ fulfils $h\perp g$, then Remark~\ref{rmk:characterizationofcomplexlinearity} yields
\begin{align*}
\Lambda&=-R_{g}\circ\Lambda\circ L_{g}\,,\\
R_{h}\circ\Lambda\circ L_{h}&=-R_{h}\circ R_{g}\circ\Lambda\circ L_{g}\circ L_{h}=-R_{gh}\circ\Lambda\circ L_{gh}\,,
\end{align*}
whence
\[\Lambda+R_{g}\circ\Lambda\circ L_{g}+R_{h}\circ\Lambda\circ L_{h}+R_{gh}\circ\Lambda\circ L_{gh}=0\,.\]
Since $(1,g,h,gh)$ is an orthonormal basis of $\hh$, Lemma~\ref{lem:characterizationofregularity} guarantees that $\Lambda\in\lin^\rr_{\mathcal{RF}}(\hh)$.

Now let us assume $\Lambda\in\lin^\rr_{\mathcal{RF}}(\hh)$ and prove properties {\it 2}, {\it 3}, {\it 4}. Corollary~\ref{cor:matrixcharacterization} yields that
\[\M_\Lambda(g,g)=\frac12\Vert\Lambda\Vert^2\langle g,g\rangle-\frac12\A_\Lambda(g,g)=\frac12\left(\Vert\Lambda\Vert^2-\A_\Lambda(g,g)\right)\,.\]
By Proposition~\ref{prop:holomorphy}, the last expression vanishes if, and only if, $\Lambda\in\lin^\cc_{gg}(\hh_l,\hh_r)$. Furthermore, Corollary~\ref{cor:formmatrix} tells us that $\M_{\Lambda}$ is positive semidefinite. Thus, $\M_{\Lambda}(g,g)=0$ is equivalent to $\begin{bmatrix} g_1&g_2&g_3\end{bmatrix}^T\in\ker(M_{\Lambda,\B})$. In particular,
\[c(\Lambda)=\dim\ker(M_{\Lambda,\B})=3-\rank(\M_\Lambda)=3-s(\Lambda)\,,\]
where the last equality follows from Remark~\ref{rmk:size}.
\end{proof}

About the last theorem, a bibliographic comment is in order: in~\cite[\S5]{tarallo}, property {\it 1} was proven with a different technique and property {\it 4} was stated without proof.

After setting the notation
\[\lin^\rr_{\mathcal{RF}}(\hh)_t:=\{\Lambda\in\lin^\rr_{\mathcal{RF}}(\hh): s(\Lambda)\leq t\}\]
for $t\in\{0,1,2,3\}$, we can write down explicit decompositions:
\begin{itemize}
\item $\lin^\rr_{\mathcal{RF}}(\hh)_2=\bigcup_{g\in\s}\lin^\cc_{gg}(\hh_l,\hh_r)$;
\item $\lin^\rr_{\mathcal{RF}}(\hh)_1=\bigcup_{g,g'\in\s,g'\neq\pm g}\lin^\cc_{gg}(\hh_l,\hh_r)\cap\lin^\cc_{g'g'}(\hh_l,\hh_r)$;
\item $\lin^\rr_{\mathcal{RF}}(\hh)_0=\{0\}=\bigcup_{g,g',g''\in\s,\langle g, g'\times g''\rangle\neq0}\lin^\cc_{gg}(\hh_l,\hh_r)\cap\lin^\cc_{g'g'}(\hh_l,\hh_r)\cap\lin^\cc_{g''g''}(\hh_l,\hh_r)$.
\end{itemize}

Complex linearity of regular real linear maps can be further studied, as follows.

\begin{theorem}\label{thm:classificationlinear}
Let $\Lambda\in\lin^\rr_{\mathcal{RF}}(\hh)$. One of the following facts is true.
\begin{enumerate}
\item $s(\Lambda)=3$. For all $g,h\in\s$, $\Lambda\not\in\lin^\cc_{gh}(\hh_l,\hh_r)$. 
\item $s(\Lambda)=2$. There exists a unique $\hat g\in\s$ such that $\Lambda\in\lin^\cc_{gh}(\hh_l,\hh_r)\Longleftrightarrow(g,h)\in\{(\hat g,\hat g),(-\hat g,-\hat g)\}$. When $\hat g$ is completed to any orthogonal basis $(e_1,e_2,\hat g)$ of $\im(\hh)$, $\Lambda$ can be represented as $\Lambda=\lambda_1\bar\vartheta_{e_1}+\lambda_2\bar\vartheta_{e_2}$ for some nonzero $\lambda_1,\lambda_2\in\hh$.
\item $s(\Lambda)=1$. For all $g\in\s$, there exists a unique $h\in\s$ such that $\lin^\cc_{gh}(\hh_l,\hh_r)$. Namely, if $\Lambda=\lambda_1\bar\vartheta_{e_1}$ for some $\lambda_1\in\hh\setminus\{0\},e_1\in\s$, then $h=\bar\vartheta_{e_1}(g)$ (which equals $g$ if, and only if, $g\perp e_1$).
\item $s(\Lambda)=0$. $\Lambda\equiv0$ is included in $\lin^\cc_{gh}(\hh_l,\hh_r)$ for all $g,h\in\s$.
\end{enumerate}
\end{theorem}

\begin{proof}
Let us fix an orthonormal basis $\B:=(e_0,e_1,e_2,e_3)$ of $\hh$, with $e_0=1$ and such that, for appropriate $\lambda_1,\lambda_2,\lambda_3\in\hh$,
\[A=A_{\Lambda,\B}=\mathrm{diag}\left(-|\lambda_1|^2+|\lambda_2|^2+|\lambda_3|^2,|\lambda_1|^2-|\lambda_2|^2+|\lambda_3|^2,|\lambda_1|^2+|\lambda_2|^2-|\lambda_3|^2\right)\,,\]
whence
\[A^TA=\mathrm{diag}\left(\left(-|\lambda_1|^2+|\lambda_2|^2+|\lambda_3|^2\right)^2,\left(|\lambda_1|^2-|\lambda_2|^2+|\lambda_3|^2\right)^2,\left(|\lambda_1|^2+|\lambda_2|^2-|\lambda_3|^2\right)^2\right)\,.\]
Set $\widehat{\B}:=(e_1,e_2,e_3)$. We separate several cases.
\begin{enumerate}
\item If $s(\Lambda)=3$, then $\lambda_m\neq0$ for all $m\in\{1,2,3\}$. Clearly, in such a case each eigenvalue of $A^TA$ is strictly less than $(|\lambda_1|^2+|\lambda_2|^2+|\lambda_3|^2)^2=\Vert\Lambda\Vert^4$. Since $A^TA$ is the matrix associated to $Q_\Lambda$ with respect to $\widehat{\B}$, it follows that  the maximum eigenvalue of $\Q_\Lambda$ is strictly less than $\Vert\Lambda\Vert^4$. By Proposition~\ref{prop:holomorphy}, there exist no $g,h\in\s$ such that $\Lambda\in\lin^\cc_{gh}(\hh_l,\hh_r)$.
\item If $s(\Lambda)=2$, then $\lambda_m=0$ for exactly one $m\in\{1,2,3\}$. Without loss of generality, we assume $\lambda_1,\lambda_2\neq0,\lambda_3=0$. In this situation,
\[A^TA=\mathrm{diag}\left(\left(-|\lambda_1|^2+|\lambda_2|^2\right)^2,\left(|\lambda_1|^2-|\lambda_2|^2\right)^2,\left(|\lambda_1|^2+|\lambda_2|^2\right)^2\right)\]
has two eigenvalues that are less than $\Vert\Lambda\Vert^4=(|\lambda_1|^2+|\lambda_2|^2)^2$ and a third eigenvalue that equals $\Vert\Lambda\Vert^4$. More precisely, $\Q_\Lambda(g)=\Vert\Lambda\Vert^4$ for $g\in\s$ if, and only if, $g=\pm e_3$. Moreover, since
\[A=\mathrm{diag}\left(-|\lambda_1|^2+|\lambda_2|^2,|\lambda_1|^2-|\lambda_2|^2,|\lambda_1|^2+|\lambda_2|^2\right)\]
is the matrix associated to $\L_\Lambda$ with respect to $\widehat{\B},\widehat{\B}$, we have $\L_\Lambda(\pm e_3)=\pm \Vert\Lambda\Vert^2 e_3$. By Proposition~\ref{prop:holomorphy}, the inclusion $\Lambda\in\lin^\cc_{gh}(\hh_l,\hh_r)$ holds if, and only if, $g=h=\pm e_3$.
\item If $s(\Lambda)=1$, then $\lambda_m=0$ for exactly two choices of $m\in\{1,2,3\}$. Without loss of generality, we assume $\lambda_1\neq0,\lambda_2=\lambda_3=0$. In this situation,
\[A^TA=|\lambda_1|^4I_3\,,\]
whence $\Q_\Lambda(g)=\Vert\Lambda\Vert^4$ for all $g\in\s$. Moreover, since
\[A=\mathrm{diag}\left(-|\lambda_1|^2,|\lambda_1|^2,|\lambda_1|^2\right)\]
is the matrix associated to $\L_\Lambda$ with respect to $\widehat{\B},\widehat{\B}$, we conclude that $\L_\Lambda=\Vert\Lambda\Vert^2\bar\vartheta_{e_1}$. By Proposition~\ref{prop:holomorphy}, the inclusion $\Lambda\in\lin^\cc_{gh}(\hh_l,\hh_r)$ holds if, and only if, $g\in\s$ and $h=\bar\vartheta_{e_1}(g)$.\qedhere
\end{enumerate}
\end{proof}

\begin{example}
Let us go through Example~\ref{ex:size1} again. For $g\in\s$, the size $1$ map $\bar\vartheta_i(z_1+z_2j)=\bar z_1+j\bar z_2$ belongs to $\lin^\cc_{gh}(\hh_l,\hh_r)$ if, and only if $h=\bar\vartheta_i(g)$; the equality $\bar\vartheta_i(g)=g$ holds if, and only if, $g=\cos(\alpha)j+\sin(\alpha)k$ for some $\alpha\in\rr$. Similar considerations apply to $\bar\vartheta_j$ and $\bar\vartheta_k$.
\end{example}

\begin{example}
Let us go through Example~\ref{ex:size2} again. The map $\Lambda_i=-\frac{i}2\bar\vartheta_j-\frac{i}2\bar\vartheta_k$, corresponding to the Fueter variable $\zeta_1=x_1-ix_0$, has size $2$. Now we also know that $\Lambda_i\in\lin^\cc_{gh}(\hh_l,\hh_r)$ exactly when $g=h=\pm i$. For the remaining Fueter variables, we similarly remark that $\Lambda_j\in\lin^\cc_{gh}(\hh_l,\hh_r)$ exactly when $g=h=\pm j$ and $\Lambda_k\in\lin^\cc_{gh}(\hh_l,\hh_r)$ exactly when $g=h=\pm k$.
\end{example}

\begin{remark}
For all $\widehat{g}\in\s$, the map $\Lambda_{\widehat{g}}\in\lin^\rr(\hh)$ defined by the formula $\Lambda_{\widehat{g}}(x):=\langle\widehat{g},x\rangle-\widehat{g}\langle 1,x\rangle$ is an element of $\lin^\rr_{\mathcal{RF}}(\hh)$ of size $2$. It belongs to $\lin^\cc_{gh}(\hh_l,\hh_r)$ exactly when $g=h=\pm\widehat{g}$.
\end{remark}

\begin{example}
We saw in Example~\ref{ex:size3} that $\Lambda(z_1+z_2j)=2z_1+\bar z_1+j\bar z_2$ is an element of $\lin^\rr_{\mathcal{RF}}(\hh)$ having size $s(\Lambda)=3$. Thus, there exist no $g,h\in\s$ such that $\Lambda\in\lin^\cc_{gh}(\hh_l,\hh_r)$.
\end{example}

We conclude this section with two more useful results about regular real linear maps, which require some preparation.

\begin{definition}\label{def:jacobian}
Fix an orthonormal basis $\B=(e_0,e_1,e_2,e_3)$ of $\hh$ with $e_0=1,e_1e_2=e_3$.

For $\Lambda\in\lin^\rr(\hh)$, we let $\jac_\B^{lr}(\Lambda)$ denote the $4\times4$ matrix with $\cc_{e_1}$ entries associated to $\Lambda$ with respect to the left $\cc_{e_1}$-basis $(e_0,e_2)$ of $\hh=\cc_{e_1}+\cc_{e_1}e_2$ and the right $\cc_{e_1}$-basis $(e_0,e_2)$ of $\hh=\cc_{e_1}+e_2\cc_{e_1}$. Equivalently, if $\Lambda(z_1+z_2e_2)=f_1(z_1,z_2)+e_2f_2(z_1,z_2)$ for $z_1,z_2\in\cc_{e_1}$, then
\begin{align*}
\jac_\B^{lr}(\Lambda)&=\left[\frac{\partial(f_1 f_2 \bar f_1 \bar f_2)}{\partial(z_1 z_2 \bar z_1 \bar z_2)}\right]:=
\begin{bmatrix}
\dfrac{\partial f_1}{\partial z_1}&\dfrac{\partial f_1}{\partial z_2}&\dfrac{\partial f_1}{\partial \bar z_1}&\dfrac{\partial f_1}{\partial \bar z_2} \\[1em]
\dfrac{\partial f_2}{\partial z_1}&\dfrac{\partial f_2}{\partial z_2}&\dfrac{\partial f_2}{\partial \bar z_1}&\dfrac{\partial f_2}{\partial \bar z_2}\\[1em]
\dfrac{\partial \bar f_1}{\partial z_1}&\dfrac{\partial \bar f_1}{\partial z_2}&\dfrac{\partial \bar f_1}{\partial \bar z_1}&\dfrac{\partial \bar f_1}{\partial \bar z_2}\\[1em]
\dfrac{\partial \bar f_2}{\partial z_1}&\dfrac{\partial\bar f_2}{\partial z_2}&\dfrac{\partial\bar f_2}{\partial \bar z_1}&\dfrac{\partial\bar f_2}{\partial \bar z_2}
\end{bmatrix}\,.
\end{align*}

For $\underline{\Lambda}\in\lin^\rr(\hh)$, we let $\jac_\B^{rl}(\underline{\Lambda})$ denote the $4\times4$ matrix with $\cc_{e_1}$ entries associated to $\underline{\Lambda}$ with respect to the right $\cc_{e_1}$-basis $(e_0,e_2)$ of $\hh=\cc_{e_1}+e_2\cc_{e_1}$ and to the left $\cc_{e_1}$-basis $(e_0,e_2)$ of $\hh=\cc_{e_1}+\cc_{e_1}e_2$. Equivalently, if $\underline{\Lambda}(w_1+e_2w_2)=g_1(w_1,w_2)+g_2(w_1,w_2)e_2$ for $w_1,w_2\in\cc_{e_1}$, then
\[\jac_\B^{rl}(\underline\Lambda)=\left[\frac{\partial(g_1 g_2 \bar g_1 \bar g_2)}{\partial(w_1 w_2 \bar w_1 \bar w_2)}\right]\,.\]
\end{definition}

We recall the following definition and property, see~\cite[\S26.1]{librogarrett}.

\begin{definition}\label{def:adjugatematrix}
Let $A$ be a square matrix over any field $\mathbb{F}$. For any indices $i,j$, let $A_{ij}$ denote the $3\times3$ matrix obtained from $A$ by erasing its $i^{\mathrm{th}}$ row and $j^{\mathrm{th}}$ column.
The \emph{adjugate matrix} of $A$ is the matrix $A^{\mathrm{adjg}}$ whose $(i,j)$-entry is $(-1)^{i+j}\det(A_{ji})$.
\end{definition}

\begin{remark}\label{rmk:adjugatematrix}
Let $A$ be an $n\times n$ matrix over any field $\mathbb{F}$. Then the product $A\,A^{\mathrm{adjg}}$ equals $\det_{\mathbb{F}}(A)$ times the $n\times n$ identity matrix. If, moreover, $A$ is invertible, then
\[A^{-1}=\left(\textstyle\det{_\mathbb{F}}(A)\right)^{-1}A^{\mathrm{adjg}}\,.\]
\end{remark}

We give one more definition and make a remark.

\begin{definition}\label{def:adjugate}
Let $\Lambda,\underline{\Lambda}\in\lin^\rr(\hh)$. If $\jac_\B^{rl}(\underline{\Lambda})$ is the adjugate matrix to $\jac_\B^{lr}(\Lambda)$, then we call $\underline{\Lambda}$ the \emph{adjugate} of $\Lambda$ and denote it by the symbol $\Lambda^{\mathrm{adjg}}$.
\end{definition}

\begin{remark}\label{rmk:adjugate}
Let $\Lambda\in\lin^\rr(\hh)$ and let $\det(\Lambda)$ denote the real determinant of $\Lambda$. As a consequence of Remark~\ref{rmk:adjugatematrix}, the equality $\Lambda\circ\Lambda^{\mathrm{adjg}}(x)=\det(\Lambda)\,x$ holds true for all $x\in\hh$. If $\rank(\Lambda)=4$, we conclude that $\Lambda^{-1}=\frac1{\det(\Lambda)}\Lambda^{\mathrm{adjg}}$. If $\rank(\Lambda)\leq2$, we remark instead that $\Lambda^{\mathrm{adjg}}=0$: indeed, if the matrix $\jac_\B^{lr}(\Lambda)$ has rank $2$ at most, then its adjugate matrix is the zero matrix. 
\end{remark}

We are now ready to state the following theorem, whose proof we postpone to the appendix.

\begin{theorem}\label{thm:adjugate}
If $\Lambda\in\lin^\rr_{\mathcal{RF}}(\hh)$, then
\[\det(\M_\Lambda)=\frac14 |\underline{\lambda}_0|^2\,,\]
where $\underline{\lambda}_0=(cj,\Lambda^{\mathrm{adjg}})_r$ or, equivalently, where $\Lambda^{\mathrm{adjg}}=\sum_{m=0}^3\bar\vartheta_{v_m}\underline{\lambda}_m$ with respect to any orthogonal basis $(v_0,v_1,v_2,v_3)$ of $\hh$ with $v_0=1$.
\end{theorem}

Theorem~\ref{thm:adjugate} has the following consequence.

\begin{corollary}\label{cor:adjugate}
For $\Lambda\in\lin^\rr_{\mathcal{RF}}(\hh)$, the following properties are equivalent:
\begin{enumerate}
\item There exist $g,h\in\s$ such that $\Lambda\in\lin^\cc_{gh}(\hh_l,\hh_r)$.
\item $s(\Lambda)\leq2$.
\item $\det(\M_\Lambda)=0$.
\item $\Lambda^{\mathrm{adjg}}\in\lin^\rr_{\mathcal{LF}}(\hh)$.
\end{enumerate}
These properties automatically hold true if $\rank(\Lambda)\leq2$, which can only happen if $\rank(\Lambda)\in\{0,2\}$. They are false if $\rank(\Lambda)=3$. When, instead, $\rank(\Lambda)=4$, then the aforementioned properties hold true if, and only if, $\Lambda$ is RL-biregular (in which case, $\Lambda$ is automatically orientation-reversing).
\end{corollary}

\begin{proof}
Properties {\it 1} and {\it 2} are equivalent by Theorem~\ref{thm:classificationlinear}. Properties {\it 2} and {\it 3} are equivalent by Remark~\ref{rmk:size}. Property {\it 3} is, in turn, equivalent to property {\it 4} by Theorem~\ref{thm:adjugate} and by the definition of $\lin^\rr_{\mathcal{LF}}(\hh)$ given in Subsection~\ref{subsec:regularlinearmaps}.

When $\rank(\Lambda)$ is odd, then property {\it 1} is false. Moreover, if $\rank(\Lambda)\leq2$, then $\Lambda^{\mathrm{adjg}}=0$ by Remark~\ref{rmk:adjugate}, whence property {\it 4} holds true. This yields the desired conclusions that all four properties hold true when $\rank(\Lambda)\in\{0,2\}$, that $\rank(\Lambda)$ cannot equal $1$ under our hypothesis $\Lambda\in\lin^\rr_{\mathcal{RF}}(\hh)$, and that none of the four properties holds true when $\rank(\Lambda)=3$. Finally: if $\rank(\Lambda)=4$, then Remark~\ref{rmk:adjugate} guarantees that $\Lambda^{-1}$ belongs to $\lin^\rr_{\mathcal{LF}}(\hh)$ if, and only if, $\Lambda^{\mathrm{adjg}}$ does.
\end{proof}

\begin{example}
Let us consider the following $\Lambda\in\lin^\rr_{\mathcal{RF}}(\hh)$:
\[\Lambda(z_1+z_2j)=z_1+\bar z_1+j\bar z_2\,,\]
i.e., $\Lambda=\bar\vartheta_i+\frac12\bar\vartheta_j+\frac12\bar\vartheta_k$. With respect to the standard basis $\B=(1,i,j,k)$, we compute
\begin{align*}
\jac_\B^{lr}(\Lambda)&=\begin{bmatrix}
1&0&1&0\\
0&0&0&1\\
1&0&1&0\\
0&1&0&0
\end{bmatrix}\,,\\
(\jac_\B^{lr}(\Lambda))^{\mathrm{adjg}}&=\begin{bmatrix}
-1&0&1&0\\
0&0&0&0\\
1&0&-1&0\\
0&0&0&0
\end{bmatrix}\,,
\end{align*}
and conclude that
\[\Lambda^{\mathrm{adjg}}(w_1+jw_2)=-w_1+\bar w_1\,.\]
Considering the decomposition $\Lambda^{\mathrm{adjg}}=\bar\vartheta_1\frac12+\bar\vartheta_i\frac12+\bar\vartheta_j\left(-\frac12\right)+\bar\vartheta_k\left(-\frac12\right)$, we see that $\Lambda^{\mathrm{adjg}}$ does not belong to $\lin^\rr_{\mathcal{LF}}(\hh)$.
Consistently, $\det(\M_\Lambda)=\frac1{16}$ and $s(\Lambda)=3$ because
\[M_{\Lambda,\B}=\begin{bmatrix}
1&0&0\\
0&1/4&0\\
0&0&1/4
\end{bmatrix}\,.\]
Thus, there exist no $g,h\in\s$ such that $\Lambda\in\lin^\cc_{gh}(\hh_l,\hh_r)$. All these properties are consistent with the fact that $\Lambda$ has rank $3$.
\end{example}


\section{New results about right Fueter-regular functions and holomorphy}\label{sec:regularfunctions}

Let us consider again a connected open subset $U$ of $\hh$ and $C^1$ functions $f,g:U\to\hh$. The following fact is known in literature (see~\cite{perottibiregular,tarallo}):

\begin{remark}\label{ex:complexholomorphic}
Fix an orthonormal basis $\B=(e_0,e_1,e_2,e_3)$ of $\hh$ with $e_0=1,e_1e_2=e_3$.

If $df_p\in\lin^\cc_{e_1e_1}(\hh_l,\hh_r)$ for all $p\in U$, then $f$ is right Fueter-regular. Indeed, let $f_1,f_2$ be the $\cc_{e_1}$-valued functions of the variables $z_1=x_0+e_1x_1,z_2=x_2+e_1x_3$ such that $f(z_1+z_2e_2)=f_1(z_1,z_2)+e_2f_2(z_1,z_2)$ whenever $z_1+z_2e_2\in U$. Then $f$ is right Fueter-regular if, and only if,
\begin{align*}
\bar\partial_\B f&=\frac{\partial f}{\partial \bar z_1}+\frac{\partial f}{\partial \bar z_2}e_2=\frac{\partial f_1}{\partial \bar z_1}-\overline{\frac{\partial f_2}{\partial \bar z_2}}+e_2\left(\frac{\partial f_2}{\partial \bar z_1}+\overline{\frac{\partial f_1}{\partial \bar z_2}}\right)
\end{align*}
vanishes identically, which is in turn equivalent to
\begin{equation}\label{eq:right Fueter-regular}
\frac{\partial f_1}{\partial \bar z_1}=\overline{\frac{\partial f_2}{\partial \bar z_2}},\qquad\frac{\partial f_2}{\partial \bar z_1}=-\overline{\frac{\partial f_1}{\partial \bar z_2}}\,.
\end{equation}
These equations are automatically fulfilled if $f_1,f_2$ are holomorphic in the variables $z_1,z_2$.

Similarly, an inclusion $dg_p\in\lin^\cc_{e_1e_1}(\hh_r,\hh_l)$ for all $p\in U$ implies that $g$ is left Fueter-regular. Indeed, let $g_1,g_2$ be the $\cc_{e_1}$-valued functions of the variables $w_1=x_0+e_1x_1,w_2=x_2-e_1x_3$ such that $g(w_1+e_2w_2)=g_1(w_1,w_2)+g_2(w_1,w_2)e_2$ whenever $w_1+e_2w_2\in U$. Then $f$ is left Fueter-regular if, and only if, $\frac{\partial g_1}{\partial \bar w_1}-\overline{\frac{\partial g_2}{\partial \bar w_2}}+\left(\frac{\partial g_2}{\partial \bar w_1}+\overline{\frac{\partial g_1}{\partial \bar w_2}}\right)e_2$ vanishes identically, which is equivalent to having $\frac{\partial g_1}{\partial \bar w_1}=\overline{\frac{\partial g_2}{\partial \bar w_2}},\frac{\partial g_2}{\partial \bar w_1}=-\overline{\frac{\partial g_1}{\partial \bar w_2}}$. These equations are automatically fulfilled if $g_1,g_2$ are holomorphic in the variables $w_1,w_2$.
\end{remark}

We can use Theorem~\ref{thm:lineargg} to prove a more general result. We recall the following conventions, set in Subsection~\ref{subsec:regularfunctions}: given a smooth orthogonal almost-complex structure $\I$ on $U$ and a smooth orthogonal almost-complex structure $\J$ on an open neighborhood $V$ of $f(U)$, we think of $\I_p$ at a single point $p\in U$ as left multiplication times an imaginary unit $I(p)\in\s$ within $\hh_l$ and we think of $\J_q$ at a single point $q\in V$ as right multiplication times an imaginary unit $J(q)\in\s$ within $\hh_r$.

\begin{proposition}
Let $U$ be a connected open subset of $\hh$, endowed with a smooth orthogonal almost-complex structure $\I$. Let $f:U\to\hh$ be a $C^1$ function and let $V$ be an open neighborhood of $f(U)$ endowed with a smooth orthogonal almost-complex structure $\J$ such that $J(f(p))=I(p)$ for all $p\in U$. If $f:(U,\I)\to(V,\J)$ is holomorphic, then $f$ is right Fueter-regular. Moreover, $s(df_p)\leq2$ for all $p\in U$.
\end{proposition}

\begin{proof}
For all $p\in U$, the hypothesis that $f:(U,\I)\to(V,\J)$ is holomorphic implies that $df_p$ belongs to $\lin^\cc_{I(p)J(f(p))}(\hh_l,\hh_r)$. Suppose $J(f(p))=I(p)$ for all $p\in U$. By Theorem~\ref{thm:lineargg}, this implies that $df_p\in\lin^\rr_{\mathcal{RF}}(\hh)_2$ for all $p\in U$. In particular, $f$ is right Fueter-regular.
\end{proof}

We now aim at classifying right Fueter-regular functions and studying their holomorphy by means of size. To this end, we will need the next lemma, which uses the material overviewed in Subsection~\ref{subsec:conformal}. We recall that we have named a quaterionic function \emph{absolutely biregular} if it fulfils all four notions of biregularity overviewed in Section~\ref{sec:introduction}.

\begin{lemma}\label{lem:conformalandregular}
Let $U$ be a nonempty connected open subset of $\hh$. Assume $f:U\to\hh$ to be both conformal and right Fueter-regular: then there exist $g\in\s$ and $\lambda,\mu\in\hh$ with $\lambda\neq0$ such that $f$ is (the restriction to $U$ of) the orientation-reversing conformal real affine transformation $\lambda\,\bar\vartheta_g + \mu=\bar\vartheta_{\lambda g}\lambda\, + \mu$ in $U$. In particular, $f$ is RL-biregular, with inverse $\bar\vartheta_{g}\lambda' + \mu'=\lambda'\bar\vartheta_{g\bar\lambda}+\mu'$, where $\lambda':=\bar\vartheta_g(\lambda^{-1})$ and $\mu':=-\bar\vartheta_g(\lambda^{-1}\mu)$. Moreover, the following are equivalent: $\lambda\perp g$; $f$ is left Fueter-regular; $f$ is absolutely biregular.
\end{lemma}

\begin{proof}
By Liouville's Theorem, two possibilities arise for the conformal function $f:U\to\hh$.
\begin{itemize}
\item $f$ is (the restriction to $U$ of) a conformal real affine transformation of $\hh$, having either the form $\vartheta_q\,\gamma+\mu$ with $\gamma,q\in\hh^*,\mu\in\hh$ or the form $\lambda\,\bar\vartheta_q +\mu$ with $\lambda,q\in\hh^*,\mu\in\hh$. The former option is excluded because $df_p\in\lin^\rr_{\mathcal{RF}}(\hh)$ for all $p\in U$, as $f$ is right Fueter-regular. Thus, $f$ coincides in $U$ with $\lambda\,\bar\vartheta_q+\mu$. The fact that $\lambda\,\bar\vartheta_q=df_p\in\lin^\rr_{\mathcal{RF}}(\hh)$ for $p\in U$ implies that $q\in\im(\hh)$. If we set $g:=\frac{q}{|q|}\in\s$, then $\bar\vartheta_q=\bar\vartheta_g$ and $f$ coincides in $U$ with $\lambda\,\bar\vartheta_g + \mu=\bar\vartheta_{\lambda g}\lambda\, + \mu$. We remark that $f$ is left Fueter-regular if, and only if, $\lambda g\in\im(\hh)$, which is in turn equivalent to $\lambda\perp g$.
\item $f$ is (the restriction to $U$ of) a conformal real affine transformation of $\hh$, precomposed with an inversion $\tau^{q_0}$ with $q_0\in\hh\setminus U$. We will prove that this is impossible when $f$ is right Fueter-regular, using the following property proven in Subsection~\ref{subsec:conformal}:
\[d\tau^{q_0}_p=-(\bar p-\bar q_0)^{-2}\bar\vartheta_{\bar p-\bar q_0}\in\lin^\hh_{\bar\vartheta_{\bar p-\bar q_0}}(\hh_l,\hh_r)\,.\]
Either $f$ coincides in $U$ with $\gamma\,\vartheta_q\circ\tau^{q_0}+\mu$ for some $\gamma,q\in\hh^*,\mu\in\hh$ or $f$ coincides in $U$ with $\lambda\,\bar\vartheta_q\circ\tau^{q_0} +\mu$ for some $\lambda,q\in\hh^*,\mu\in\hh$. In the former case, we conclude that
\[df_p\in\lin^\hh_{\vartheta_q\circ\bar\vartheta_{\bar p-\bar q_0}}(\hh_l,\hh_r)=\lin^\hh_{\bar\vartheta_{q(\bar p-\bar q_0)}}(\hh_l,\hh_r)\,.\]
The image of $U$ through the nonconstant real affine map $p\mapsto q(\bar p-\bar q_0)$ is an open subset of $\hh$, whence not a subset of $\im(\hh)$. Therefore, it is impossible for $df_p$ to be orthogonal to $cj=\bar\vartheta_1$, i.e., to belong to $\lin^\rr_{\mathcal{RF}}(\hh)$, for all $p\in U$. In the latter case, we have
\[df_p\in\lin^\hh_{\bar\vartheta_q\circ\bar\vartheta_{\bar p-\bar q_0}}(\hh_l)=\lin^\hh_{\vartheta_{q(\bar p-\bar q_0)}}(\hh_l)\,,\]
which also makes it impossible for $df_p$ to belong to $\lin^\rr_{\mathcal{RF}}(\hh)$ for all $p\in U$.
In either case, we have reached a contradiction with the hypothesis that $f$ is right Fueter-regular.
\end{itemize}
We are left with studying the inverse $f^{-1}$ of $f=\lambda\,\bar\vartheta_g+\mu$. Since $\bar\vartheta_g^{-1}=\bar\vartheta_g$, we have
\begin{align*}
f^{-1}(y)&=\bar\vartheta_g(\lambda^{-1}(y-\mu))= g (\bar y - \bar\mu)\bar \lambda^{-1} \bar g\\
&=g \bar y \bar \lambda^{-1}\bar g - g \bar\mu \bar \lambda^{-1}\bar g=g \bar y \bar g g \bar \lambda^{-1}\bar g - \bar\vartheta_g(\lambda^{-1}\mu)\\
&=\bar\vartheta_g(y) \bar\vartheta_g(\lambda^{-1}) + \mu' = \bar\vartheta_g(y) \lambda' + \mu'
\end{align*}
for all $y\in f(U)$. Thus, $f^{-1}$ is automatically left Fueter-regular. Moreover, $f^{-1}=\bar\vartheta_g\lambda'+\mu'=\lambda'\,\bar\vartheta_{(\lambda')^{-1}g}+\mu'$ is right-Fueter-regular if, and only if,
\[\im(\hh)\ni(\lambda')^{-1}g=\bar\vartheta_g(\lambda^{-1})^{-1}g=(g\bar\lambda^{-1}\bar g)^{-1}g=g\bar\lambda\,,\]
which is equivalent to $\lambda\perp g$.
\end{proof}

We are now ready for the announced result, which classifies right Fueter-regular functions according to the sizes of their real differentials and studies their holomorphy. The theorem subsumes~\cite[Theorem 6.1]{tarallo}, within cases {\it 1} and {\it 2}.

\begin{theorem}\label{thm:classificationfueter}
Let $U$ be a nonempty connected open subset of $\hh$ and let $f:U\to\hh$ be a right Fueter-regular function. One of the following properties holds:
\begin{enumerate}
\item $f$ is a constant $\mu\in\hh$ and $s(df_p)=0=\rank(df_p)$ for all $p\in U$.
\item $f$ is an  RL-biregular orientation-reversing conformal real affine transformation; more precisely, there exist $g\in\s$ and $\lambda,\mu\in\hh$ with $\lambda\neq0$ such that $f$ coincides with $\lambda\,\bar\vartheta_g + \mu=\bar\vartheta_{\lambda g}\lambda\, + \mu$ in $U$. In particular: for all $p\in U$, $df_p=\lambda\,\bar\vartheta_g$ has $s(df_p)=1$ and $\rank(df_p)=4$. If, moreover, $\lambda\perp g$, then $f$ is absolutely biregular.
\item There exists a proper real analytic subset $E$ of $U$ such that: for all $p\in U\setminus E$, the equalities $s(df_p)=2$ and $\rank(df_p)\in\{2,4\}$ hold; for all $p\in E$, either $s(df_p)=1$ and $\rank(df_p)=4$ or $s(df_p)=0=\rank(df_p)$. Moreover, if $p\in U$ is such that $\rank(df_p)=4$, then $f$ is locally RL-biregular and orientation-reversing near $p$.
\item There exists a proper real analytic subset $E$ of $U$ such that $s(df_p)=3$ and $\rank(df_p)\in\{3,4\}$ for all $p\in U\setminus E$. Moreover, $f$ is not locally RL-biregular near any point of $U$.
\end{enumerate}
In particular, in cases 3 and 4, the Hausdorff dimension of $E$ is less than, or equal, to $3$. In each case listed above, respectively:
\begin{enumerate}
\item[1'.] For every smooth orthogonal almost-complex structure $\I$ on $U$ and every open neighborhood $V$ of $f(U)$ endowed with a smooth orthogonal almost-complex structure $\J$, the map $f:(U,\I)\to(V,\J)$ is holomorphic.
\item[2'.] For every smooth orthogonal almost-complex structure $\I$ on $U$, there exists a unique smooth orthogonal almost-complex structure $\J$ on the open set $f(U)$ such that the map $f:(U,\I)\to(f(U),\J)$ is biholomorphic; namely, $\J$ is determined by the equality $J(f(p))=\bar\vartheta_g(I(p))$, valid for all $p\in U$. Furthermore, for each $p\in U$, the equality $J(f(p))=I(p)$ is equivalent to $I(p)\perp g$.
\item[3'.] For every contractible open subset $C$ of $U\setminus E$, there exists a real analytic map $I:C\to\s$ (inducing a smooth orthogonal almost-complex structure $\I$ on $C$) such that
\[df_p\in\lin^\cc_{gh}(\hh_l,\hh_r)\Longleftrightarrow g=h=\pm I(p)\]
for all $p\in C$. As a consequence: if $U'\subseteq C$ is an open set endowed with a map $I':U'\to\s$ inducing a smooth orthogonal almost-complex structure $\I'$ on $U'$; if $V'$ is an open neighborhood of $f(U')$ endowed with a map $J':V'\to\s$ inducing a smooth orthogonal almost-complex structure $\J'$ on $V'$; then $f:(U',\I')\to(V',\J')$ is holomorphic if, and only if, $I'(p)=J'(f(p))=\pm I(p)$ for all $p\in U'$. In particular, in the special case when $f:C\to f(C)$ is a diffeomorphism, setting $J(q):=I(f^{-1}(q))$ for all $q\in f(C)$ defines the only smooth orthogonal almost-complex structure $\J$ making $f:(C,\I)\to(f(C),\J)$ holomorphic.
\item[4'.] For every $p\in U\setminus E$ and all $g,h\in\s$, we have $df_p\not\in\lin^\cc_{gh}(\hh_l,\hh_r)$. In particular: for every smooth orthogonal almost-complex structure $\I$ on any open subset $U'\subseteq U$ and every open neighborhood $V'$ of $f(U')$ endowed with a smooth orthogonal almost-complex structure $\J$, the map $f:(U',\I)\to(V',\J)$ is not holomorphic.
\end{enumerate}
\end{theorem}

\begin{proof}
Since $f$ is right Fueter-regular, it is real analytic. Moreover, for all $p\in U$, $df_p\in\lin^\rr_{\mathcal{RF}}(\hh)$. For $s\in\{0,1,2,3\}$, let $U_s$ denote the set of points $p\in U$ that fulfil the equality $s(df_p)=s$, which is the same as $\rank(\M_{df_p})=s$. Clearly, $U=U_3\cup U_2\cup U_1\cup U_0$.

The subset $U_3$ of $U$ is open. If $U_3$ is not empty, then $E=U_0\cup U_1\cup U_2$ is a proper real analytic subset of $U$: namely, the zero set of the real analytic function $U\to\rr,\ p\mapsto\det(\M_{df_p})$. Taking into account Corollary~\ref{cor:adjugate}, property {\it 4} holds. By Theorem~\ref{thm:classificationlinear}, the equality $s(df_p)=3$ (valid for all $p\in U\setminus E$) implies that $df_p\not\in\lin^\cc_{gh}(\hh_l,\hh_r)$ for all $g,h\in\s$. Property {\it 4'} immediately follows.

If $U_3=\emptyset$, then $U_2$ is an open subset of $U$. Let us set $\B:=(1,i,j,k)$. If $U_2$ is not empty, then $E=U_0\cup U_1$ is a proper real analytic subset of $U$: namely, $E$ is the set of common zeros of the nine functions $U\to\rr$ mapping each $p\in U$ into the determinants of the $2\times2$ minors of the matrix $M_{df_p,\B}$. Moreover, Theorem~\ref{thm:classificationlinear} guarantees that the rank of $df_p$ equals $0$ at all $p\in U_0$ and equals $4$ at all $p\in U_1$. Taking into account Corollary~\ref{cor:adjugate}, property {\it 3} follows. The equality $s(df_p)=2$ (valid for all $p\in U\setminus E$) implies that there exists a unique $\widehat{g}_p\in\s$ such that the inclusion $df_p\in\lin^\cc_{gh}(\hh_l,\hh_r)$ is equivalent to $g=h=\pm\widehat{g}_p$. According to Theorem~\ref{thm:lineargg}, $\pm\widehat{g}_p$ are the unitary elements of $\ker(M_{df_p,\B})$. Since $f$ is a real analytic function, for each $p\in U\setminus E$ there exist an open ball $B_p\subseteq U\setminus E$ centered at $p$ and a real analytic map $I_p:B_p\to\s$ such that $I_p(x)\in\{\pm\widehat{g}_x\}$ for all $x\in B_p$. Now, let 
\[\pi:E\to M_{3,3}^2(\rr):=\{A\in M_{3,3}(\rr) : \rank(A)=2\}\]
be the $\rr^*$-principal bundle over $M_{3,3}^2(\rr)$ whose fiber over $A\in M_{3,3}^2(\rr)$ is $\pi^{-1}(A)=(\ker(A)\setminus\{0\})\times\{A\}$. Given any contractible open subset $C$ of $U\setminus E$, consider the map $\Phi:C\to M_{3,3}^2(\rr),\ x\mapsto M_{df_x,\B}$ and the pullback bundle $\Phi^*E\to C$. Since $C$ is contractible, standard results such as~\cite[Corollary 11.6]{librosteenrod} guarantee that $\Phi^*E\to C$ is trivial and has a global continuous section $C\to\Phi^*E,\ x\mapsto(v(x),M_{df_x,\B})$. We may define a continuous map $I:C\to\s$ by choosing $I(x)$ to be the quaternionic imaginary unit corresponding to the normalized vector $\frac{v(x)}{\Vert v(x)\Vert}\in\rr^3$. For any $p\in C$, the continuous map $I$ coincides near $p$ with either the real analytic map $I_p:B_p\to\s$ or its opposite. Thus, $I$ itself is real analytic. If, for some open subset $U'$ of $C$ and some open neighborhood $V'$ of $f(U')$ the map $f:(U',\I')\to(V',\J')$ is holomorphic, then for all $p\in U'$ we have $df_p\in\lin^\cc_{I'(p),J'(f(p))}(\hh_l,\hh_r)$, whence $I'(p)=J'(f(p))=\pm I(p)$. Finally, in the special case when $f:C\to f(C)$ is a diffeomorphism, the real analytic map $J:f(C)\to\s,\ q\mapsto I(f^{-1}(q))$ is the only smooth map $J:f(C)\to\s$ such that $I(p)=J(f(p))$ for all $p\in C$. Property {\it 3'} is therefore proven.

If $U_2=U_3=\emptyset$, then $U_1$ is an open subset of $U$. If it is not empty, then $f_{|_{U_1}}$ is not only a right Fueter-regular map but also, by Remark~\ref{rmk:size1isconformal}, a conformal map. By Lemma~\ref{lem:conformalandregular}, there exist $g\in\s$ and $\lambda,\mu\in\hh$ with $\lambda\neq0$ such that $f$ coincides with $\lambda \bar\vartheta_g + \mu=\bar\vartheta_{\lambda g}\lambda+\mu$ in $U_1$. Since $f:U\to\hh$ is real analytic, the same is true in $U$. In particular, $U_1=U$. Thus, property {\it 2} holds. By Theorem~\ref{thm:classificationlinear}, the equality $s(df_p)=1$ (valid for all $p\in U$) implies that, for each $I\in\s$, the inclusion $df_p\in\lin^\cc_{IJ}(\hh_l,\hh_r)$ is equivalent to $J=\bar\vartheta_g(I)$. The latter expression equals $I$ if, and only if, $I\perp g$. Property {\it 2'} immediately follows.

If $U_1=U_2=U_3=\emptyset$, then $U_0=U$. In other words, $df_p=0$ for all $p\in U$. Thus, the real analytic function $f$ is constant in the connected open set $U$, which is property {\it 1}. By Theorem~\ref{thm:classificationlinear}, the equality $s(df_p)=0$ (valid for all $p\in U$) implies that, for each $g,h\in\s$, the inclusion $df_p\in\lin^\cc_{gh}(\hh_l,\hh_r)$ holds. Property {\it 1'} immediately follows.
\end{proof}

In particular, we have found yet another proof of a well-known property of right Fueter-regular functions $f:U\to\hh$: for all $p\in U$, the rank of $df_p$ is $0,2,3,$ or $4$. We have also proven the analog, for right Fueter-regular functions, of~\cite[Theorem 5]{perottibiregular}.

We illustrate Theorem~\ref{thm:classificationfueter} with some examples, which also show that all combinations of rank and size envisioned in the theorem are possible. Examples of properties {\it 1} and {\it 1'} and of properties {\it 2} and {\it 2'} of Theorem~\ref{thm:classificationfueter} are easy to find.

\begin{example}
Consider the real affine transformations
\begin{align*}
f&:\hh\to\hh,\quad z_1+z_2j\mapsto z_2+j(z_1-i)\,,\\
g&:\hh\to\hh,\quad z_1+z_2j\mapsto iz_1+1+j(iz_2)\,.
\end{align*}
Since $\bar\vartheta_j(z_1+z_2j)=z_1-jz_2$, these transformations can be expressed as
\begin{align*}
f&=j\bar\vartheta_j+k=\bar\vartheta_1j+k\,,\\
g&=i\bar\vartheta_j+1=\bar\vartheta_ki+1\,.
\end{align*}
In particular, $f,g$ are RL-biregular orientation-reversing conformal real affine transformations of $\hh$. We point out that $g$ is absolutely biregular, while $f$ is not left Fueter-regular. For every smooth orthogonal almost-complex structure $\I$ on $\hh$, with $\I_p=L_{I(p)}$ for all $p\in\hh$, there exists a unique smooth orthogonal almost-complex structure $\J$ on $\hh$ such that the maps $f,g:(\hh,\I)\to(\hh,\J)$ are biholomorphic; namely, $\J_p:=R_{\bar\vartheta_j(I(p))}$ for all $p\in\hh$. Furthermore, for each $p\in\hh$, we have that $\bar\vartheta_j(I(p))=jI(p)j$ equals $I(p)$ if, and only if, $I(p)\in i\rr+k\rr$.
\end{example}

Our next two examples concern properties {\it 3} and {\it 3'} of Theorem~\ref{thm:classificationfueter}.

\begin{example}
Consider the right (and left) Fueter-regular function $f:\hh\to\cc\subset\hh$ defined by the formula $f(z_1+z_2j)=\phi(z_1)$, where $\phi:\cc\to\cc$ is a nonconstant entire function. The real differential of $f$ can be expressed as
\[df_{z_1+z_2j}=\phi'(z_1)dz_1=\frac{\phi'(z_1)}2\bar\vartheta_j+\frac{\phi'(z_1)}2\bar\vartheta_k\,.\]
The matrix $M=M_{\Lambda,\B}$, with $\Lambda=df_{z_1+z_2j}$ and $\B=(1,i,j,k)$, can be easily computed by means of Theorem~\ref{thm:formmatrix}:
\[M=|\phi'(z_1)|^2\begin{bmatrix}
0&0&0\\
0&1/4&0\\
0&0&1/4\\
\end{bmatrix}\,.\]
If we set $E:=\{z_1+z_2j\in\hh : \phi'(z_1)=0\}$, then the real differential of $f$ has size $0$ (and rank $0$) at all points of $E$ and size $2$ (and rank $2$) at all points of $\hh\setminus E$. At each $p\in\hh\setminus E$, we have $df_p\in\lin^\cc_{gh}(\hh_l,\hh_r)$ if, and only if, $g=h=\pm i$. For any connected open set $U$, $\pm L_i$ are the only smooth orthogonal almost-complex structures $\I$ on $U$ such that $df_p\circ\I_p=R_{I(p)}\circ df_p$ for all $p\in U$. For any open $V\supset f(U)$ and any smooth orthogonal almost-complex structure $\J$ on $V$, the map $f:(U,L_i)\to(V,\J)$ is holomorphic if, and only if, $\J_{|_{f(U)}}=R_i$. Both $\J=R_i$ and $\J_p=R_{J(p)}$ with
\[J(w_1+jw_2):=\frac{i+jw_2}{\sqrt{1+|w_2|^2}}\,,\] are examples of smooth orthogonal almost-complex structures $\J$ on $\hh$ such that $f:(\hh\setminus E,L_i)\to(\hh,\J)$ is holomorphic.
\end{example}

Within the last example, both constant structures $L_i$ and $R_i$ are integrable on $\hh$, while the nonconstant structure $\J$ defined on $\hh$ is not integrable, see~\cite{wood}.

\begin{example}\label{ex:biregularsize2}
Consider the function
\[f:\hh\to\hh,\quad z_1+z_2j\mapsto\bar z_1+z_2^2+j\bar z_2\,,\]
i.e., $f(x_0+ix_1+jx_2+kx_3)=x_0+x_2^2-x_3^2+i(2x_2x_3-x_1)+jx_2+kx_3$. It clearly is a diffeomorphism, with inverse $f^{-1}(y_0+iy_1+jy_2+ky_3)=y_0-y_2^2+y_3^2+i(2y_2y_3-y_1)+jy_2+ky_3$, i.e., $f^{-1}(w_1+jw_2)=\bar w_1-w_2^2+\bar w_2 j$. The real differential of $f$, namely 
\begin{align*}
df_{z_1+z_2j}&=0 dz_1+ 2z_2 dz_2+d\bar z_1+jd\bar z_2\,,
\end{align*}
has rank $4$ and can be expressed as
\[df_{z_1+z_2j}=\bar\vartheta_i+z_2j\bar\vartheta_j-z_2j\bar\vartheta_k\,.\]
Thus, $f$ is right Fueter-regular.

For $z_2=0$, we have $df_{z_1}=\bar\vartheta_i$, whence $s(df_{z_1})=1$ and $df_{z_1}\in\lin^\cc_{gg}(\hh_l,\hh_r)$ for all $g\in\s$ with $g\perp i$. We assume henceforth $z_2\neq0$ and compute the matrix $M=M_{\Lambda,\B}$, with $\Lambda=df_{z_1+z_2j}$ and $\B=(1,i,j,k)$, using Theorem~\ref{thm:formmatrix}. Splitting $z_2$ into real coordinates as $z_2=x_2+ix_3$, we obtain
\begin{align*}
&\langle 1i,z_2jj\rangle=-\langle i,z_2\rangle=-x_3\,,\\
&\langle 1i,-z_2jk\rangle=-\langle i,z_2i\rangle=-\langle 1,z_2\rangle=-x_2\,,\\
&\langle z_2jj,-z_2jk\rangle=\langle z_2,z_2i\rangle=0\,,
\end{align*}
whence
\[M=\begin{bmatrix}
1&-x_3&-x_2\\
-x_3&|z_2|^2&0\\
-x_2&0&|z_2|^2
\end{bmatrix}\,.\]
By direct computation, the eigenvalues of $M$ are $1+|z_2|^2,|z_2|^2,0$, while the corresponding eigenspaces are spanned by $\begin{bmatrix}1&-x_3&-x_2\end{bmatrix}^T,\begin{bmatrix}0&x_2&-x_3\end{bmatrix}^T,\begin{bmatrix}|z_2|^2&x_3&x_2\end{bmatrix}^T$, respectively.

On the one hand, we conclude that $f$ is RL-biregular and that the only $g\in\s$ such that $df_{z_1+z_2j}\in\lin^\cc_{gg}(\hh_l,\hh_r)$ are
\[g=\pm\frac{|z_2|^2i+x_3j+x_2k}{\sqrt{|z_2|^4+|z_2|^2}}=\pm\frac{|z_2|^2i+kz_2}{\sqrt{|z_2|^4+|z_2|^2}}\,.\]
Thus, setting for all $z_1,z_2\in\cc$ with $z_2\neq0$
\[I(z_1+z_2j):=\frac{|z_2|^2i+kz_2}{\sqrt{|z_2|^4+|z_2|^2}}=:J(f(z_1+z_2j))\]
as well as $\I_{z_1+z_2j}:=L_{I(z_1+z_2j)}, \J_{w_1+jw_2}:=R_{J(w_1+jw_2)}$ defines two smooth orthogonal almost-complex structures $\I,\J$ on $\hh\setminus\cc$ such that $f: (\hh\setminus\cc, \I)\to(\hh\setminus\cc,\J)$ is holomorphic.

On the other hand, we find that $df_{z_1+z_2j}$ has size $s(df_{z_1+z_2j})=2$. We also find, in terms of the vectors $i-x_3j-x_2k=i-kz_2$ and $x_2j-x_3k=jz_2$, a decomposition 
\begin{align*}
df_{z_1+z_2j}&=\lambda_1\bar\vartheta_{i-kz_2}+\lambda_2\bar\vartheta_{jz_2}\,,\\
\lambda_1&=(df_{z_1+z_2j},\bar\vartheta_{i-kz_2})_l=1-jz_2\,,\\
\lambda_2&=(df_{z_1+z_2j},\bar\vartheta_{jz_2})_l=jz_2\,.
\end{align*}
\end{example}

For the orthogonal complex structure $\I$ constructed in the last example, an explicit computation of the Nijenhuis tensor at a point $z_1+z_2j$ yields $N^3_{12}=\frac{x_3}{|z_2|^2(1+|z_2|^2)^2},N^4_{12}=\frac{-x_2}{|z_2|^2(1+|z_2|^2)^2}$. We conclude that $\I$ is not integrable. Similarly, $\J$ is not integrable. Our next example concerns properties {\it 4} and {\it 4'}.

\begin{example}
Consider the function
\[f:\hh\to\rr+j\rr+k\rr\subset\hh,\quad z_1+z_2j\mapsto |z_1|^2-|z_2|^2+j\bar z_1\bar z_2\,.\]
The real differential of $f$, namely 
\begin{align*}
df_{z_1+z_2j}&=\bar z_1dz_1-\bar z_2 dz_2+(z_1+j\bar z_2)d\bar z_1+(-z_2+j\bar z_1)d\bar z_2\\
&=\bar z_1dz_1-\bar z_2 dz_2+(z_1+z_2j)d\bar z_1+(z_1+z_2j)jd\bar z_2
\end{align*}
can be expressed as
\[df_{z_1+z_2j}=(z_1+z_2j)\bar\vartheta_i+\frac{\bar z_1-\bar z_2j}{2}\bar\vartheta_j+\frac{\bar z_1+\bar z_2j}{2}\bar\vartheta_k\,.\]
Thus, $f$ is right Fueter-regular. For $z_1+z_2j=0$, we have $df_0=0$, whence $s(df_0)=0=\rank(df_0)$ and $df_0\in\lin^\cc_{gh}(\hh_l,\hh_r)$ for all $g,h\in\s$. We assume henceforth $z_1+z_2j\neq0$ and compute the matrix $M=M_{\Lambda,\B}$, with $\Lambda=df_{z_1+z_2j}$ and $\B=(1,i,j,k)$, using Theorem~\ref{thm:formmatrix}:
\[M=(|z_1|^2+|z_2|^2)\begin{bmatrix}
1&0&0\\
0&1/4&0\\
0&0&1/4
\end{bmatrix}\,.\]
Since $M$ has rank $3$, Theorem~\ref{thm:lineargg} yields the equality $s(df_{z_1+z_2j})=3$. Moreover, $\rank(df_{z_1+z_2j})=3$. By Theorem~\ref{thm:classificationfueter}, $f$ is not holomorphic. More precisely, for all $p\in\hh\setminus\{0\}$ and all $g,h\in\s$ we have $df_p\not\in\lin^\cc_{gh}(\hh_l,\hh_r)$.
\end{example}

In the last example, we could have concluded that properties {\it 4} and {\it 4'} applied by simply proving that $\rank(df_p)=3$ at all $p\in\hh\setminus\{0\}$. In case a right Fueter-regular function $f:U\to\hh$ has $\rank(df_p)=4$ at some point $p\in U$, a simple criterion to determine whether properties {\it 4} and {\it 4'} apply to $f$ is the following.

\begin{theorem}\label{thm:differentialcriterionholomorphy}
Fix a connected open subset $U$ of $\hh$ and a right Fueter-regular function $f:U\to\hh$. Fix an orthonormal basis $\B=(e_0,e_1,e_2,e_3)$ of $\hh$ with $e_0=1$ and $e_1e_2=e_3$. Let us split the quaternionic variable of $f$ as $z_1+z_2e_2$, with $z_1,z_2$ in $\cc_{e_1}$ and $f$ as $f(z_1+z_2e_2)=f_1(z_1,z_2)+e_2f_2(z_1,z_2)$. If $p\in U$, then there exist $g,h\in\s$ such that $df_p\in\lin^\cc_{gh}(\hh_l,\hh_r)$ if, and only if, the expressions
{\small
\begin{align}
&\left(\left|\frac{\partial f_1}{\partial z_1}\right|^2-\left|\frac{\partial f_2}{\partial z_2}\right|^2\right)\frac{\partial f_1}{\partial \bar z_1}+\left(\overline{\frac{\partial f_2}{\partial z_1}}\frac{\partial f_1}{\partial z_1}+\frac{\partial f_1}{\partial z_2}\overline{\frac{\partial f_2}{\partial z_2}}\right)\frac{\partial f_2}{\partial \bar z_1}-\left(\overline{\frac{\partial f_2}{\partial z_1}}\frac{\partial f_2}{\partial z_2}+\overline{\frac{\partial f_1}{\partial z_1}}\frac{\partial f_1}{\partial z_2}\right)\overline{\frac{\partial f_2}{\partial \bar z_1}}\,,
\label{eq:differentialcriterionholomorphy1}\\
&\left(\frac{\partial f_2}{\partial z_1}\overline{\frac{\partial f_2}{\partial z_2}}+\frac{\partial f_1}{\partial z_1}\overline{\frac{\partial f_1}{\partial z_2}}\right)\frac{\partial f_1}{\partial \bar z_1}+\left(\overline{\frac{\partial f_2}{\partial z_1}}\frac{\partial f_1}{\partial z_1}+\frac{\partial f_1}{\partial z_2}\overline{\frac{\partial f_2}{\partial z_2}}\right)\overline{\frac{\partial f_1}{\partial \bar z_1}}-\left(\left|\frac{\partial f_1}{\partial z_2}\right|^2-\left|\frac{\partial f_2}{\partial z_1}\right|^2\right)\overline{\frac{\partial f_2}{\partial \bar z_1}}
\label{eq:differentialcriterionholomorphy2}
\end{align}}
both vanish at $p$.
\end{theorem}

We postpone the proof of Theorem~\ref{thm:differentialcriterionholomorphy} to the appendix and provide an example.

\begin{example}
The function
\[f:\hh\to\hh,\quad z_1+z_2j\mapsto z_1+z_2^2+\overline{z}_1+j(z_1^2+z_2+\overline{z}_2)\,,\]
has
\[\frac{\partial f_1}{\partial z_1}=\frac{\partial f_1}{\partial \overline{z}_1}=\frac{\partial f_2}{\partial z_2}=\frac{\partial f_2}{\partial \overline{z}_2}\equiv1,\frac{\partial f_1}{\partial z_2}=2z_2,\frac{\partial f_2}{\partial z_1}=2z_1,
\frac{\partial f_1}{\partial \overline{z}_2}=\frac{\partial f_2}{\partial \overline{z}_1}\equiv0\,.\]
It is right Fueter-regular because conditions~\eqref{eq:right Fueter-regular} are fulfilled. Condition~\eqref{eq:differentialcriterionholomorphy1} is fulfilled throughout $\hh$, but condition~\eqref{eq:differentialcriterionholomorphy2} is only fulfilled in the hyperplane $\Pi:z_1+\overline{z}_1+z_2+\overline{z}_2=0$, i.e., $\Pi:x_0+x_2=0$, of $\hh$. Therefore, $f$ is not holomorphic. More precisely,  for all $p\in\hh\setminus\Pi$ and all $g,h\in\s$, we have $df_p\not\in\lin^\cc_{gh}(\hh_l,\hh_r)$. Moreover, $f$ is not locally RL-biregular despite being a local diffeomorphism from an open subset $U$ of $\hh$ to $f(U)$. For instance: at $p_0=\frac{1+j}2\not\in\Pi$, the real differential
\[df_{p_0}=dz_1+dz_2+d\overline{z}_1+j(dz_1+dz_2+d\overline{z}_2)\]
has rank $4$ and corresponds to the element $\Lambda(z_1+z_2j)=z_1+z_2+\bar z_1+j(z_1+z_2+\bar z_2)$ of $\lin^\rr_{\mathcal{RF}}(\hh)$, i.e.,
\[\Lambda(x_0+ix_1+jx_2+kx_3)=2x_0+x_2+ix_3+j(x_0+2x_2)+k(-x_1)\,,\]
whose inverse
\[\Lambda^{-1}(y_0+iy_1+jy_2+yx_3)=\frac23y_0-\frac13y_2-iy_3+j\left(\frac23y_2-\frac13y_0\right)+ky_1\,,\]
i.e., $\Lambda^{-1}(w_1+jw_2)=\frac13(w_1+w_2+\bar w_1-2\bar w_2)+\frac13(w_1+w_2-2\bar w_1+\bar w_2)j$, is an element of $\lin^\rr(\hh)$ that does not belong to $\lin^\rr_{\mathcal{LF}}(\hh)$.
\end{example}


\section*{Appendix}
\setcounter{section}{0}
\setcounter{theorem}{0}

We conclude our work with a technical appendix, which serves to prove Theorem~\ref{thm:adjugate} and Theorem~\ref{thm:differentialcriterionholomorphy}. We begin with three useful lemmas.

\begin{lemma}\label{lem:complexifiedLambda1}
Let $\Lambda\in\lin^\rr(\hh)$. Fix an orthonormal basis $\B=(e_0,e_1,e_2,e_3)$ of $\hh$ with $e_0=1$ and $e_1e_2=e_3$. There exist quaternions $m_1,m_2,n_1,n_2$ such that
\begin{align*}
\Lambda(z_1+z_2e_2) = m_1z_1+m_2z_2+n_1\bar z_1+n_2\bar z_2\,,
\end{align*}
for all $z_1,z_2\in\cc_{e_1}$. Moreover, if $\Lambda=\sum_{m=0}^3\lambda_m\bar\vartheta_{e_m}$, then
\[\lambda_0=\frac12(n_1+n_2e_2), \lambda_1=\frac12(n_1-n_2e_2),\lambda_2=\frac12(m_1+m_2e_2), \lambda_3=\frac12(m_1-m_2e_2)\,.\]
In particular:
\begin{enumerate}
\item $\Lambda\in\lin^\rr_{\mathcal{RF}}(\hh)$ if, and only if, $n_2=n_1e_2$. If this is the case, then $\lambda_1=n_1$. Moreover,
\begin{align}\label{eq:hermitianM}
M_{\Lambda,\B}=\begin{bmatrix}
\Vert c\Vert_H^2&\frac12\re(e_1c,a-b)_H&\frac12\re(c,a+b)_H\\
\frac12\re(e_1c,a-b)_H&\frac14\Vert a-b\Vert_H^2&\frac14\re(-a+b, (a+b)e_1)_H\\
\frac12\re(c,a+b)_H&\frac14\re(-a+b, (a+b)e_1)_H&\frac14\Vert a+b\Vert_H^2\\
\end{bmatrix}\,,
\end{align}
where $(\cdot,\cdot)_H$ denotes the standard Hermitian product on $\cc_{e_1}^2$, where $\Vert\cdot\Vert_H$ denotes the corresponding complex norm and where $a=(a_1,a_2),b=(b_1,b_2),c=(c_1,c_2)\in\cc_{e_1}^2$ are the coordinates of $m_1e_2,m_2,n_1$ with respect to the right $\cc_{e_1}$-basis $(e_0,e_2)$ of $\hh=\cc_{e_1}+e_2\cc_{e_1}$; equivalently, where
\[\jac_\B^{lr}(\Lambda)=\begin{bmatrix}
\overline{a}_2&b_1&c_1&-\overline{c}_2\\
-\overline{a}_1&b_2&c_2&\overline{c}_1\\
\overline{c}_1&-c_2&a_2&\overline{b}_1\\
\overline{c}_2&c_1&-a_1&\overline{b}_2
\end{bmatrix}\,.\]
\item $\Lambda\in\lin^\cc_{e_1,e_1}(\hh_l,\hh_r)$ if, and only if, $n_1=0=n_2$. If this is the case, then $\lambda_0=0=\lambda_1$ and
\begin{align*}
M_{\Lambda,\B}=\begin{bmatrix}
0&0&0\\
0&\frac14\Vert a-b\Vert_H^2&\frac14\re(-a+b, (a+b)e_1)_H\\
0&\frac14\re(-a+b, (a+b)e_1)_H&\frac14\Vert a+b\Vert_H^2\\
\end{bmatrix}\,,
\end{align*}
\end{enumerate}
\end{lemma}

\begin{proof}
By direct inspection,
\begin{align*}
z_1&=\frac{1}2\bar\vartheta_{e_2}+\frac{1}2\bar\vartheta_{e_3}\,,\quad
z_2=\frac{e_2}2\bar\vartheta_{e_2}-\frac{e_2}2\bar\vartheta_{e_3}\,,\\
\bar z_1&=\frac{1}2\bar\vartheta_{e_0}+\frac{1}2\bar\vartheta_{e_1}\,,\quad
\bar z_2=\frac{e_2}2\bar\vartheta_{e_0}-\frac{e_2}2\bar\vartheta_{e_1}.
\end{align*}
Since $\bar\vartheta_{e_0},\bar\vartheta_{e_1},\bar\vartheta_{e_2},\bar\vartheta_{e_3}$ form a left $\hh$-basis of $\lin^\rr(\hh)$, so do the maps $z_1+z_2e_2\mapsto z_1$, $z_1+z_2e_2\mapsto z_2$, $z_1+z_2e_2\mapsto\bar z_1$, and $z_1+z_2e_2\mapsto \bar z_2$. Thus, the desired quaternions $m_1,m_2,n_1,n_2$ exist. Moreover, taking into account that $(\bar\vartheta_{e_s},\bar\vartheta_{e_t})_l=0$ whenever $s\neq t$, we have
\begin{align*}
\lambda_0&=(\Lambda,\bar\vartheta_{e_0})_l=n_1(\bar z_1,\bar\vartheta_{e_0})_l+n_2(\bar z_2,\bar\vartheta_{e_0})_l=\frac{n_1}2+\frac{n_2e_2}2\,,\\
\lambda_1&=(\Lambda,\bar\vartheta_{e_1})_l=n_1(\bar z_1,\bar\vartheta_{e_1})_l+n_2(\bar z_2,\bar\vartheta_{e_1})_l=\frac{n_1}2-\frac{n_2e_2}2\,,\\
\lambda_2&=(\Lambda,\bar\vartheta_{e_2})_l=m_1(z_1,\bar\vartheta_{e_2})_l+m_2(z_2,\bar\vartheta_{e_2})_l=\frac{m_1}2+\frac{m_2e_2}2\,,\\
\lambda_3&=(\Lambda,\bar\vartheta_{e_3})_l=m_1(z_1,\bar\vartheta_{e_3})_l+m_2(z_2,\bar\vartheta_{e_3})_l=\frac{m_1}2-\frac{m_2e_2}2\,,
\end{align*}
as desired. In the special case when $\Lambda\in\lin^\rr_{\mathcal{RF}}(\hh)$, the equality $\lambda_0=0$ yields $n_2=n_1e_2$. Corollary~\ref{cor:formmatrix} implies that
\begin{align*}
&M_{\Lambda,\B}=\\
&\begin{bmatrix}
|n_1|^2&\frac12\langle n_1 e_1, m_1e_2-m_2\rangle&\frac12\langle n_1, m_1e_2+m_2\rangle\\
\frac12\langle n_1 e_1, m_1e_2-m_2\rangle&\frac14|m_1e_2-m_2|^2&\frac14\langle -m_1e_2+m_2, (m_1e_2+m_2)e_1\rangle\\
\frac12\langle n_1, m_1e_2+m_2\rangle&\frac14\langle -m_1e_2+m_2, (m_1e_2+m_2)e_1\rangle&\frac14|m_1e_2+m_2|^2\\
\end{bmatrix}
\end{align*}
Taking into account the equalities $\langle u_1+e_2u_2,v_1+e_2v_2\rangle=\re\left((u_1,u_2),(v_1,v_2)\right)_H$ and $|u_1+e_2u_2|=\Vert(u_1,u_2)\Vert_H$, valid for all $u_1,u_2,v_1,v_2\in\cc_{e_1}$, the thesis follows.
\end{proof}

\begin{lemma}\label{lem:identitiesdiagonal}
Under the hypotheses of Lemma~\ref{lem:complexifiedLambda1} and under the additional hypothesis that $\Lambda\in\lin^\rr_{\mathcal{RF}}(\hh)$, the matrix $M_{\Lambda,\B}$ is diagonal if, and only if, the following conditions are fulfilled:
\begin{align*}
\overline{a}_1c_1+\overline{a}_2c_2+b_1\overline{c}_1+b_2\overline{c}_2&=0\,,\\
a_1\overline{b}_1-\overline{a}_1b_1+a_2\overline{b}_2-\overline{a}_2b_2&=0\,.
\end{align*}
If this is the case, then $\det(\M_\Lambda)=\frac1{16}\Vert a+b\Vert_H^2\,\Vert a-b\Vert_H^2\,\Vert c\Vert_H^2$.
\end{lemma}

\begin{proof}
By direct inspection in Lemma~\ref{lem:complexifiedLambda1}, under the additional hypothesis that $\Lambda\in\lin^\rr_{\mathcal{RF}}(\hh)$, the matrix $M_{\Lambda,\B}$ is diagonal if, and only if, the following three conditions are fulfilled:
\begin{align*}
0&=2\re(e_1c,a-b)_H=2\re\left(e_1c_1(\overline{a}_1-\overline{b}_1)+e_1c_2(\overline{a}_2-\overline{b}_2)\right)\\
&=2\re\left(e_1\left(\overline{a}_1c_1+\overline{a}_2c_2-\overline{b}_1c_1-\overline{b}_2c_2\right)\right)\\
&=e_1\left(\overline{a}_1c_1+\overline{a}_2c_2-\overline{b}_1c_1-\overline{b}_2c_2-a_1\overline{c}_1-a_2\overline{c}_2+b_1\overline{c}_1+b_2\overline{c}_2\right)\,,\\
0&=2\re(c,a+b)_H=2\re\left(c_1(\overline{a}_1+\overline{b}_1)+c_2(\overline{a}_2+\overline{b}_2)\right)\\
&=2\re\left(\overline{a}_1c_1+\overline{a}_2c_2+\overline{b}_1c_1+\overline{b}_2c_2\right)\\
&=\overline{a}_1c_1+\overline{a}_2c_2+\overline{b}_1c_1+\overline{b}_2c_2+a_1\overline{c}_1+a_2\overline{c}_2+b_1\overline{c}_1+b_2\overline{c}_2\,,\\
0&=\re(-a+b, (a+b)e_1)_H=\re\left((a_1-b_1)(\overline{a}_1+\overline{b}_1)e_1+(a_2-b_2)(\overline{a}_2+\overline{b}_2)e_1\right)\\
&=\re\left(\left(a_1\overline{b}_1-\overline{a}_1b_1+a_2\overline{b}_2-\overline{a}_2b_2\right)e_1\right)\\
&=-a_1\overline{b}_1+\overline{a}_1b_1-a_2\overline{b}_2+\overline{a}_2b_2\,.
\end{align*}
The first two equalities are equivalent to the mutually equivalent equalities
\begin{align*}
0&=\overline{a}_1c_1+\overline{a}_2c_2+b_1\overline{c}_1+b_2\overline{c}_2\,,\\
0&=\overline{b}_1c_1+\overline{b}_2c_2+a_1\overline{c}_1+a_2\overline{c}_2\,.
\end{align*}
This proves the first statement. Another direct inspection in Lemma~\ref{lem:complexifiedLambda1} proves the second statement.
\end{proof}

\begin{lemma}\label{lem:complexifiedLambda2}
Let $\underline{\Lambda}\in\lin^\rr(\hh)$. Fix an orthonormal basis $\B=(e_0,e_1,e_2,e_3)$ of $\hh$ with $e_0=1$ and $e_1e_2=e_3$. There exist quaternions $\underline{m}_1,\underline{m}_2,\underline{n}_1,\underline{n}_2$ such that
\begin{align*}
\underline{\Lambda}(w_1+e_2w_2) = w_1\underline{m}_1+w_2\underline{m}_2+\bar w_1\underline{n}_1+\bar w_2\underline{n}_2\,,
\end{align*}
for all $w_1,w_2\in\cc_{e_1}$. Moreover, if $\underline{\Lambda}=\sum_{m=0}^3\bar\vartheta_{e_m}\underline{\lambda}_m$, then
\[\underline{\lambda}_0=\frac12(\underline{n}_1+e_2\underline{n}_2), \underline{\lambda}_1=\frac12(\underline{n}_1-e_2\underline{n}_2),\underline{\lambda}_2=\frac12(\underline{m}_1+e_2\underline{m}_2), \underline{\lambda}_3=\frac12(\underline{m}_1-e_2\underline{m}_2)\,.\]
\end{lemma}

\begin{proof}
By direct inspection,
\begin{align*}
\bar\vartheta_{e_0}(w_1+e_2w_2)&=\bar w_1-\bar w_2e_2\,,\quad
\bar\vartheta_{e_1}(w_1+e_2w_2)=\bar w_1+\bar w_2e_2\,,\\
\bar\vartheta_{e_2}(w_1+e_2w_2)&=w_1-w_2 e_2\,,\quad
\bar\vartheta_{e_3}(w_1+e_2w_2)=w_1+w_2e_2\,,
\end{align*}
whence
\begin{align*}
w_1&=\bar\vartheta_{e_2}\frac{1}2+\bar\vartheta_{e_3}\frac{1}2\,,\quad
w_2=\bar\vartheta_{e_2}\frac{e_2}2-\bar\vartheta_{e_3}\frac{e_2}2\,,\\
\bar w_1&=\bar\vartheta_{e_0}\frac{1}2+\bar\vartheta_{e_1}\frac{1}2\,,\quad
\bar w_2=\bar\vartheta_{e_0}\frac{e_2}2-\bar\vartheta_{e_1}\frac{e_2}2.
\end{align*}
Since $\bar\vartheta_{e_0},\bar\vartheta_{e_1},\bar\vartheta_{e_2},\bar\vartheta_{e_3}$ form a right $\hh$-basis of $\lin^\rr(\hh)$, so do the maps $w_1+e_2w_2\mapsto w_1$, $w_1+e_2w_2\mapsto w_2$, $w_1+e_2w_2\mapsto\bar w_1$, and $w_1+e_2w_2\mapsto \bar w_2$. Thus, the desired quaternions $\underline{m}_1,\underline{m}_2,\underline{n}_1,\underline{n}_2$ exist. Moreover, taking into account that $(\bar\vartheta_{e_s},\bar\vartheta_{e_t})_r=0$ whenever $s\neq t$, we have
\begin{align*}
\underline{\lambda}_0&=(\bar\vartheta_{e_0},\underline{\Lambda})_r=(\bar\vartheta_{e_0},\bar w_1)_r\underline{n}_1+(\bar\vartheta_{e_0},\bar w_2)_r\underline{n}_2=\frac12\underline{n}_1+\frac{e_2}2\underline{n}_2\,,\\
\underline{\lambda}_1&=(\bar\vartheta_{e_1},\underline{\Lambda})_r=(\bar\vartheta_{e_1},\bar w_1)_r\underline{n}_1+(\bar\vartheta_{e_1},\bar w_2)_r\underline{n}_2=\frac12\underline{n}_1-\frac{e_2}2\underline{n}_2\,,\\
\underline{\lambda}_2&=(\bar\vartheta_{e_2},\underline{\Lambda})_r=(\bar\vartheta_{e_2},w_1)_r\underline{m}_1+(\bar\vartheta_{e_2},w_2)_r\underline{m}_2=\frac12\underline{m}_1+\frac{e_2}2\underline{m}_2\,,\\
\underline{\lambda}_3&=(\bar\vartheta_{e_3},\underline{\Lambda})_r=(\bar\vartheta_{e_3},w_1)_r\underline{m}_1+(\bar\vartheta_{e_3},w_2)_r\underline{m}_2=\frac12\underline{m}_1-\frac{e_2}2\underline{m}_2\,,
\end{align*}
as desired. 
\end{proof}

We are now ready to prove Theorem~\ref{thm:adjugate}.

\begin{proof}[Proof of Theorem~\ref{thm:adjugate}]
Since $\Lambda\in\lin^\rr_{\mathcal{RF}}(\hh)$, Corollary~\ref{cor:formmatrix} applies to $\Lambda$. Thus, there exists an orthonormal basis $\B=(e_0,e_1,e_2,e_3)$ of $\hh$, with $e_0=1$ and $e_1e_2=e_3$, such that $M_{\Lambda,\B}$ is diagonal. Let us set $A:=\jac_\B^{lr}(\Lambda)$: we have
\[A=\begin{bmatrix}
\overline{a}_2&b_1&c_1&-\overline{c}_2\\
-\overline{a}_1&b_2&c_2&\overline{c}_1\\
\overline{c}_1&-c_2&a_2&\overline{b}_1\\
\overline{c}_2&c_1&-a_1&\overline{b}_2
\end{bmatrix}\]
for appropriate $a=(a_1,a_2),b=(b_1,b_2),c=(c_1,c_2)\in\cc_{e_1}^2$. By Definition~\ref{def:adjugate}, $\Lambda^{\mathrm{adjg}}$ is the real linear endomorphism of $\hh$ with $\jac_\B^{rl}(\Lambda^{\mathrm{adjg}})=A^{\mathrm{adjg}}=(\underline{a}_{ij})_{i,j=1}^4$ with $\underline{a}_{ij}=(-1)^{i+j}\det(A_{ji})$ (see Definition~\ref{def:adjugatematrix}). Lemma~\ref{lem:complexifiedLambda2} explains how to find quaternions $\underline{\lambda}_0,\underline{\lambda}_1,\underline{\lambda}_1,\underline{\lambda}_3$ such that $\Lambda^{\mathrm{adjg}}=\sum_{m=0}^3\bar\vartheta_{e_m}\underline{\lambda}_m$: in particular, $\underline{\lambda}_0=\frac12(\underline{n}_1+e_2\underline{n}_2)$, assuming $\underline{m}_1,\underline{m}_2,\underline{n}_1,\underline{n}_2$ to be quaternions such that
\[\Lambda^{\mathrm{adjg}}(w_1+e_2w_2) = w_1\underline{m}_1+w_2\underline{m}_2+\bar w_1\underline{n}_1+\bar w_2\underline{n}_2\,.\]
By direct inspection, $\underline{n}_1=\underline{a}_{13}+\underline{a}_{23}e_2$ and $\underline{n}_2=\underline{a}_{14}+\underline{a}_{24}e_2$. Thus,
\begin{align*}
4|\underline{\lambda}_0|^2&=|\underline{n}_1+e_2\underline{n}_2|^2=|\underline{a}_{13}-\overline{\underline{a}}_{24}|^2+|\underline{a}_{23}+\overline{\underline{a}}_{14}|^2\\
&=\left|\det(A_{31})-\overline{\det(A_{42})}\right|^2+\left|\det(A_{32})+\overline{\det(A_{41})}\right|^2\,.
\end{align*}
Now,
\begin{align*}
\det(A_{31})-\overline{\det(A_{42})}&=\det\begin{bmatrix}
b_1&c_1&-\overline{c}_2\\
b_2&c_2&\overline{c}_1\\
c_1&-a_1&\overline{b}_2
\end{bmatrix}-
\det\begin{bmatrix}
a_2&\overline{c}_1&-c_2\\
-a_1&\overline{c}_2&c_1\\
c_1&\overline{a}_2&b_1\\
\end{bmatrix}\\
&=b_1(c_2\overline{b}_2+\overline{c}_1a_1)-b_2(c_1\overline{b}_2-\overline{c}_2a_1)+c_1\Vert c\Vert_H^2\\
&-a_2(\overline{c}_2b_1-c_1\overline{a}_2)-a_1(\overline{c}_1b_1+c_2\overline{a}_2)-c_1\Vert c\Vert_H^2\\
&=(|a_2|^2-|b_2|^2)c_1+b_1\overline{(b_2\overline{c}_2)}-a_1(\overline{a}_2c_2)+(a_1b_2-a_2b_1)\overline{c}_2\\
&=(|a_2|^2-|b_2|^2)c_1-b_1(a_1\overline{c}_1+a_2\overline{c}_2+\overline{b}_1c_1)+a_1(\overline{a}_1c_1+b_1\overline{c}_1+b_2\overline{c}_2)\\
&+(a_1b_2-a_2b_1)\overline{c}_2\\
&=(|a_2|^2-|b_2|^2)c_1+(|a_1|^2-|b_1|^2)c_1+2(a_1b_2-a_2b_1)\overline{c}_2\\
&=(\Vert a\Vert_H^2-\Vert b\Vert_H^2)c_1+2(a_1b_2-a_2b_1)\overline{c}_2\\
\det(A_{32})+\overline{\det(A_{41})}&=\det\begin{bmatrix}
\overline{a}_2&c_1&-\overline{c}_2\\
-\overline{a}_1&c_2&\overline{c}_1\\
\overline{c}_2&-a_1&\overline{b}_2
\end{bmatrix}+
\det\begin{bmatrix}
\overline{b}_1&\overline{c}_1&-c_2\\
\overline{b}_2&\overline{c}_2&c_1\\
-\overline{c}_2&\overline{a}_2&b_1
\end{bmatrix}\\
&=\overline{a}_2(c_2\overline{b}_2+\overline{c}_1a_1)+\overline{a}_1(c_1\overline{b}_2-\overline{c}_2a_1)+\overline{c}_2\Vert c\Vert_H^2\\
&+\overline{b}_1(\overline{c}_2b_1-c_1\overline{a}_2)-\overline{b}_2(\overline{c}_1b_1+c_2\overline{a}_2)-\overline{c}_2\Vert c\Vert_H^2\\
&=(\overline{a}_1\overline{b}_2-\overline{a}_2\overline{b}_1)c_1+\overline{a}_2\overline{(\overline{a}_1c_1)}-\overline{b}_2(b_1\overline{c}_1)+(|b_1|^2-|a_1|^2)\overline{c}_2\\
&=(\overline{a}_1\overline{b}_2-\overline{a}_2\overline{b}_1)c_1-\overline{a}_2(a_2\overline{c}_2+\overline{b}_1c_1+\overline{b}_2c_2)+\overline{b}_2(\overline{a}_1c_1+\overline{a}_2c_2+b_2\overline{c}_2)\\
&+(|b_1|^2-|a_1|^2)\overline{c}_2\\
&=2(\overline{a}_1\overline{b}_2-\overline{a}_2\overline{b}_1)c_1+(|b_2|^2-|a_2|^2)\overline{c}_2+(|b_1|^2-|a_1|^2)\overline{c}_2\\
&=2(\overline{a}_1\overline{b}_2-\overline{a}_2\overline{b}_1)c_1+(\Vert b\Vert_H^2-\Vert a\Vert_H^2)\overline{c}_2\,,
\end{align*}
where the fourth and tenth equalities follow from Lemma~\ref{lem:identitiesdiagonal}. We set $\beta:=2(a_1b_2-a_2b_1)$ and note, for future reference, that
\begin{align*}
|\beta|^2&=4|a_1|^2|b_2|^2+4|a_2|^2|b_1|^2-8\re(a_1\overline{a}_2\overline{b}_1b_2)\\
&=4\Vert a\Vert_H^2\Vert b\Vert_H^2-4|a_1|^2|b_1|^2-4|a_2|^2|b_2|^2-8\re(a_1\overline{a}_2\overline{b}_1b_2)\\
&=4\Vert a\Vert_H^2\Vert b\Vert_H^2-4|a_1\overline{b}_1+a_2\overline{b}_2|^2\\
&=4\Vert a\Vert_H^2\Vert b\Vert_H^2-4\re^2(a_1\overline{b}_1+a_2\overline{b}_2)\,,
\end{align*}
where the last equality follows from the inclusion $a_1\overline{b}_1+a_2\overline{b}_2\in\rr$ (see Lemma~\ref{lem:identitiesdiagonal}). We compute
\begin{align*}
4|\underline{\lambda}_0|^2&=\left|\det(A_{31})-\overline{\det(A_{42})}\right|^2+\left|\det(A_{32})+\overline{\det(A_{41})}\right|^2\\
&=\left|(\Vert a\Vert_H^2-\Vert b\Vert_H^2)c_1+\beta\overline{c}_2\right|^2+\left|\overline{\beta}c_1+(\Vert b\Vert_H^2-\Vert a\Vert_H^2)\overline{c}_2\right|^2\\
&=((\Vert a\Vert_H^2-\Vert b\Vert_H^2)^2+|\beta|^2)\Vert c\Vert_H^2+2(\Vert a\Vert_H^2-\Vert b\Vert_H^2)\re(\overline{\beta}c_1c_2)+2(\Vert b\Vert_H^2-\Vert a\Vert_H^2)\re(\overline{\beta}c_1c_2)\\
&=\left((\Vert a\Vert_H^2-\Vert b\Vert_H^2)^2+4\Vert a\Vert_H^2\Vert b\Vert_H^2-4\re^2(a_1\overline{b}_1+a_2\overline{b}_2)\right)\Vert c\Vert_H^2\\
&=\left((\Vert a\Vert_H^2+\Vert b\Vert_H^2)^2-4\re^2(a_1\overline{b}_1+a_2\overline{b}_2)\right)\Vert c\Vert_H^2\\
&=\left((|a_1|^2+|b_1|^2+|a_2|^2+|b_2|^2)^2-(2\re(a_1\overline{b}_1)+2\re(a_2\overline{b}_2))^2\right)\Vert c\Vert_H^2\\
&=\left((|a_1|^2+|b_1|^2+2\re(a_1\overline{b}_1)+|a_2|^2+|b_2|^2+2\re(a_2\overline{b}_2)\right)\\
&\cdot\left(|a_1|^2+|b_1|^2-2\re(a_1\overline{b}_1)+|a_2|^2+|b_2|^2-2\re(a_2\overline{b}_2)\right)\Vert c\Vert_H^2\\
&=(|a_1+b_1|^2+|a_2+b_2|^2)(|a_1-b_1|^2+|a_2-b_2|^2)\Vert c\Vert_H^2\\
&=\Vert a+b\Vert_H^2\,\Vert a-b\Vert_H^2\,\Vert c\Vert_H^2\\
&=16 \det(\M_\Lambda)\,,
\end{align*}
where the last equality follows from Lemma~\ref{lem:identitiesdiagonal}. The thesis immediately follows.
\end{proof}

As a byproduct of the previous proof, we can state the analog, for right Fueter-regular functions, of~\cite[Proposition 4]{perottibiregular}.

\begin{proposition}\label{prop:abc}
If $\Lambda\in\lin^\rr_{\mathcal{RF}}(\hh)$ has
\[\jac_\B^{lr}(\Lambda)=\begin{bmatrix}
\overline{a}_2&b_1&c_1&-\overline{c}_2\\
-\overline{a}_1&b_2&c_2&\overline{c}_1\\
\overline{c}_1&-c_2&a_2&\overline{b}_1\\
\overline{c}_2&c_1&-a_1&\overline{b}_2
\end{bmatrix}\,,\]
then
\begin{align*}
\det(\M_\Lambda)&=\frac1{16}\left|(|a_2|^2-|b_2|^2)c_1+(-a_1\overline{a}_2+b_1\overline{b}_2)c_2+(a_1b_2-a_2b_1)\overline{c}_2\right|^2\\
&+\frac1{16}\left|(\overline{a}_1\overline{b}_2-\overline{a}_2\overline{b}_1)c_1+(a_1\overline{a}_2-b_1\overline{b}_2)\overline{c}_1+(|b_1|^2-|a_1|^2)\overline{c}_2\right|^2\,.
\end{align*}
As a consequence: there exist $g,h\in\s$ such that $\Lambda\in\lin^\cc_{gh}(\hh_l,\hh_r)$ if, and only if,
\begin{align*}
(|a_2|^2-|b_2|^2)c_1+(-a_1\overline{a}_2+b_1\overline{b}_2)c_2+(a_1b_2-a_2b_1)\overline{c}_2&=0\,,\\
(\overline{a}_1\overline{b}_2-\overline{a}_2\overline{b}_1)c_1+(a_1\overline{a}_2-b_1\overline{b}_2)\overline{c}_1+(|b_1|^2-|a_1|^2)\overline{c}_2&=0\,.
\end{align*}
\end{proposition}

\begin{proof}
Theorem~\ref{thm:adjugate} tells us that $\det(\M_\Lambda)=\frac14|\underline{\lambda}_0|^2$. Going through the proof of Theorem~\ref{thm:adjugate} without the assumption that $M_{\Lambda,\B}$ be diagonal, we find that
\begin{align*}
4|\underline{\lambda}_0|^2&=\left|\det(A_{31})-\overline{\det(A_{42})}\right|^2+\left|\det(A_{32})+\overline{\det(A_{41})}\right|^2\\
&=\left|(|a_2|^2-|b_2|^2)c_1+(-a_1\overline{a}_2+b_1\overline{b}_2)c_2+(a_1b_2-a_2b_1)\overline{c}_2\right|^2\\
&+\left|(\overline{a}_1\overline{b}_2-\overline{a}_2\overline{b}_1)c_1+(a_1\overline{a}_2-b_1\overline{b}_2)\overline{c}_1+(|b_1|^2-|a_1|^2)\overline{c}_2\right|^2\,.
\end{align*}
The thesis immediately follows.
\end{proof}

The next remark will allow us to prove Theorem~\ref{thm:differentialcriterionholomorphy}.

\begin{remark}\label{rmk:complexjacobian}
Fix a connected open subset $U$ of $\hh$ and a $C^1$ function $f:U\to\hh$. Fix an orthonormal basis $\B=(e_0,e_1,e_2,e_3)$ of $\hh$ with $e_0=1$ and $e_1e_2=e_3$. Let us split the quaternionic variable of $f$ as $z_1+z_2e_2$, with $z_1,z_2$ in $\cc_{e_1}$ and $f$ as $f(z_1+z_2e_2)=f_1(z_1,z_2)+e_2f_2(z_1,z_2)$. Lemma~\ref{lem:complexifiedLambda1} applies to $df_p$ at every point $p\in U$ and yields that $df=\sum_{m=0}^3\lambda_m\bar\vartheta_{e_m}$, where
\begin{align*}
2\lambda_0&=\frac{\partial f_1}{\partial \bar z_1}-\overline{\frac{\partial f_2}{\partial \bar z_2}}+e_2\left(\frac{\partial f_2}{\partial \bar z_1}+\overline{\frac{\partial f_1}{\partial \bar z_2}}\right)\,,\\
2\lambda_1&=\frac{\partial f_1}{\partial \bar z_1}+\overline{\frac{\partial f_2}{\partial \bar z_2}}+e_2\left(\frac{\partial f_2}{\partial \bar z_1}-\overline{\frac{\partial f_1}{\partial \bar z_2}}\right)\,,\\
2\lambda_2&=\frac{\partial f_1}{\partial z_1}-\overline{\frac{\partial f_2}{\partial z_2}}+e_2\left(\frac{\partial f_2}{\partial z_1}+\overline{\frac{\partial f_1}{\partial z_2}}\right)\,,\\
2\lambda_3&=\frac{\partial f_1}{\partial z_1}+\overline{\frac{\partial f_2}{\partial z_2}}+e_2\left(\frac{\partial f_2}{\partial z_1}-\overline{\frac{\partial f_1}{\partial z_2}}\right)\,.
\end{align*}
In the case when $f$ is regular, which is equivalent to $\frac{\partial f_1}{\partial \bar z_1}=\overline{\frac{\partial f_2}{\partial \bar z_2}}$ and $\frac{\partial f_2}{\partial \bar z_1}=-\overline{\frac{\partial f_1}{\partial \bar z_2}}$, then equality~\eqref{eq:hermitianM} holds with
\begin{align*}
a:=\left(-\overline{\frac{\partial f_2}{\partial z_1}},\overline{\frac{\partial f_1}{\partial z_1}}\right),\ 
b:=\left(\frac{\partial f_1}{\partial z_2},\frac{\partial f_2}{\partial z_2}\right),\ 
c:=\left(\frac{\partial f_1}{\partial \bar z_1},\frac{\partial f_2}{\partial \bar z_1}\right)\in\cc_{e_1}^2\,.
\end{align*}
Finally, $df_p\in\lin^\cc_{e_1,e_1}(\hh_l,\hh_r)$, which is equivalent to $\frac{\partial f_1}{\partial \bar z_1}=\frac{\partial f_1}{\partial \bar z_2}=\frac{\partial f_2}{\partial \bar z_1}=\frac{\partial f_2}{\partial \bar z_2}\equiv0$, implies $df_p=\lambda_2\bar\vartheta_{e_2}+\lambda_3\bar\vartheta_{e_3}$.
\end{remark}

\begin{proof}[Proof of Theorem~\ref{thm:differentialcriterionholomorphy}]
The thesis follows immediately from Proposition~\ref{prop:abc} if we take into account Remark~\ref{rmk:complexjacobian}.
\end{proof}






\end{document}